\documentclass[12pt,a4paper,reqno]{amsart} 

\usepackage[T1]{fontenc}       

\usepackage{amsfonts}          

\usepackage{amsmath}

\usepackage{amssymb}

\usepackage{overpic}
\usepackage{euscript} 
\usepackage{eurosym}
\usepackage{comment}
\usepackage{mathrsfs} 
\usepackage{ifthen}
\usepackage{esint} 
\usepackage{color}

\long\def\symbolfootnote[#1]#2{\begingroup%
\def\thefootnote{\fnsymbol{footnote}}\footnote[#1]{#2}\endgroup}

{\qed\vspace{5pt}}

\usepackage{amsthm}

\newtheoremstyle{lause}
{5pt}
{5pt}
{\slshape}
{\parindent}
{\bfseries}
{.}
{.5em}
{}

\theoremstyle{lause}

\newtheoremstyle{maaritelma}
{5pt}
{5pt}
{\rmfamily}
{\parindent}
{\bfseries}
{.}
{.5em}
{}

\theoremstyle{maaritelma}
\newtheoremstyle{lause}
{5pt}
{5pt}
{\slshape}
{\parindent}
{\bfseries}
{.}
{.5em}
{}

\theoremstyle{lause}
\newtheorem{theorem}{Theorem}[section]
\newtheorem{lemma}[theorem]{Lemma}

\newtheorem{corollary}[theorem]{Corollary}
\newtheorem{problem}[theorem]{Problem}

\newtheoremstyle{maaritelma}
{5pt}
{5pt}
{\rmfamily}
{\parindent}
{\bfseries}
{.}
{.5em}
{}

\theoremstyle{maaritelma}
\newtheorem{definition}[theorem]{Definition}

\newtheorem{example}[theorem]{Example}
\newtheorem{remark}[theorem]{Remark}

\newtheorem{open}[theorem]{Open question}

\numberwithin{equation}{section}

\DeclareMathOperator*{\essinf}{ess\,inf}

\begin{document}

\thispagestyle{empty}

\begin{center}

{\large{\textbf{Fractional harmonic measure in minimum Riesz energy problems with external fields}}}

\vspace{18pt}

\textbf{Natalia Zorii}

\vspace{18pt}


\footnotesize{\address{Institute of Mathematics, Academy of Sciences
of Ukraine, Tereshchenkivska~3, 01601,
Kyiv-4, Ukraine\\
natalia.zorii@gmail.com }}

\end{center}

\vspace{12pt}

{\footnotesize{\textbf{Abstract.} For the Riesz kernel $\kappa_\alpha(x,y):=|x-y|^{\alpha-n}$ on $\mathbb R^n$, where $n\geqslant2$, $\alpha\in(0,2]$, and $\alpha<n$, we consider the problem of minimizing the Gauss functional
\[\int\kappa_\alpha(x,y)\,d(\mu\otimes\mu)(x,y)+2\int f_{q,z}\,d\mu,\quad\text{where $f_{q,z}:=-q\int\kappa_\alpha(\cdot,y)\,d\varepsilon_z(y)$},\]
$q$ being a positive number, $\varepsilon_z$ the unit Dirac measure at $z\in\mathbb R^n$, and $\mu$ ranging all probability measures of finite energy, concentrated on quasiclosed $A\subset\mathbb R^n$. For any
$z\in A^u\cup(\mathbb R^n\setminus{\rm Cl}_{\mathbb R^n}A)$, where $A^u$ is the set of all inner $\alpha$-ultrairregular points for $A$ (the concept of $\alpha$-ultrairregularity being newly introduced), we provide necessary and sufficient conditions for the existence of the minimizer $\lambda_{A,f_{q,z}}$, establish its alternative characterizations, and describe its support, thereby discovering new interesting phenomena. In detail, $z\in\partial_{\mathbb R^n}A$ is said to be inner $\alpha$-ul\-t\-ra\-irregular if the inner $\alpha$-har\-mo\-n\-ic measure $\varepsilon_z^A$ of $A$ is of finite energy. We show that for any $z\in A^u\cup(\mathbb R^n\setminus{\rm Cl}_{\mathbb R^n}A)$, $\lambda_{A,f_{q,z}}$ exists if and only if either $A$ is of finite inner capacity, or $q\geqslant H_z$, where $H_z:=1/\varepsilon_z^A(\mathbb R^n)\in[1,\infty)$. Thus, for any closed $A$, any $z\in A^u$, and any $q\geqslant H_z$~--- even arbitrarily large, no compensation effect occurs between the two oppositely signed charges, $-q\varepsilon_z$ and $\lambda_{A,f_{q,z}}$, carried by the same conductor $A$, which at the first glance seems to contradict our physical intuition. Another interesting phenomenon is that, if closed $A$ has unbounded boundary, then for any $z\in A^u\cup(\mathbb R^n\setminus A)$, the support of $\lambda_{A,f_{H_z,z}}$ is noncompact, whereas that of $\lambda_{A,f_{H_z+\beta,z}}$ is already compact for any $\beta\in(0,\infty)$~--- even arbitrarily small. The results obtained substantially improve some of the latest ones on the problem in question, e.g.\ those by Dragnev et al.\ (Constr.\ Approx., 2023), and are illustrated by means of some examples.
}}
\symbolfootnote[0]{\quad 2010 Mathematics Subject Classification: Primary 31C15.}
\symbolfootnote[0]{\quad Key words: Minimum Riesz energy problems with external fields; inner Riesz balayage; inner Riesz equilibrium measure; inner $\alpha$-har\-mon\-ic measure; inner $\alpha$-reg\-u\-l\-ar, $\alpha$-ir\-r\-eg\-ul\-ar, and $\alpha$-ul\-t\-ra\-ir\-reg\-ular points; inner $\alpha$-thin\-n\-ess and $\alpha$-ul\-t\-ra\-thin\-ness of a set at infinity.
}

\vspace{6pt}

\markboth{\emph{Natalia Zorii}} {\emph{Fractional harmonic measure in minimum Riesz energy problems with external fields}}

\section{Introduction and statement of the problem}\label{sec-alpha} This paper deals with problems of minimizing the energy of positive Radon measures $\mu$ on $\mathbb R^n$, $n\geqslant2$, with respect to the $\alpha$-Riesz kernel $\kappa_\alpha(x,y):=|x-y|^{\alpha-n}$, where $\alpha\in(0,n)$ and $\alpha\leqslant2$, evaluated in the presence of the external field $f_{q,z}$ created by an attractive mass $q\in(0,\infty)$ placed at a fixed point $z\in\mathbb R^n$. That is,
\begin{equation*}
f_{q,z}(y):=-q|z-y|^{\alpha-n},\quad y\in\mathbb R^n,
\end{equation*}
and hence the energy in question (usually referred to as {\it the Gauss functional} or, in constructive function theory, {\it the $f_{q,z}$-wei\-g\-h\-ted energy}) is given by the formula
\begin{equation}\label{ext2}I_{f_{q,z}}(\mu):=\int\kappa_\alpha(x,y)\,d(\mu\otimes\mu)(x,y)+2\int f_{q,z}(y)\,d\mu(y),\end{equation}
provided that the value on right is well defined (as a finite number or $\pm\infty$).

The main tool in this research is the inner $\alpha$-har\-mon\-ic measure $\varepsilon_z^A$ of $A\subset\mathbb R^n$, along with the closely related concept of inner $\alpha$-Riesz equilibrium measure $\gamma_A$ treated in an extended sense where its total mass as well as its energy might be infinite.

To define admissible measures in the minimum $f_{q,z}$-wei\-g\-h\-ted energy problem we are interested in, we recall some basic facts of the theory of $\alpha$-Riesz potentials \cite{L}.

We denote by $\mathfrak M$ the linear space of all re\-al-val\-u\-ed (Radon) measures $\mu$ on $\mathbb R^n$, equipped with the {\it vague} topology of pointwise convergence on the class $C_0(\mathbb R^n)$ of all continuous functions $\varphi:\mathbb R^n\to\mathbb R$ of compact support, and by $\mathfrak M^+$ the cone of all positive $\mu\in\mathfrak M$, where $\mu$ is {\it positive} if and only if $\mu(\varphi)\geqslant0$ for all positive $\varphi\in C_0(\mathbb R^n)$.

For any $\mu,\nu\in\mathfrak M$, the {\it mutual energy} $I(\mu,\nu)$ and the {\it potential} $U^\mu$ are given by
\begin{align*}
 I(\mu,\nu)&:=\int\kappa_\alpha(x,y)\,d(\mu\otimes\nu)(x,y),\\
 U^\mu(x)&:=\int\kappa_\alpha(x,y)\,d\mu(y),\quad x\in\mathbb R^n,
\end{align*}
respectively, provided that the value on right is well defined. For $\mu=\nu$, the mutual energy $I(\mu,\nu)$ defines the {\it energy} $I(\mu):=I(\mu,\mu)$.

When speaking of a positive measure $\mu\in\mathfrak M^+$, we understand that its potential $U^\mu$ is not identically infinite on $\mathbb R^n$, or equivalently \cite[Section~I.3.7]{L},
\begin{equation*}
\int_{|y|>1}\,\frac{d\mu(y)}{|y|^{n-\alpha}}<\infty.
\end{equation*}
This necessarily holds if $\mu$ is {\it bounded} (that is, with $\mu(\mathbb R^n)<\infty$), or of finite energy (see \cite[Corollary to Lemma~3.2.3]{F1}). Actually, then (and only then) the potential $U^\mu$ of any {\it signed} measure $\mu\in\mathfrak M$ is well defined and finite qua\-si-ev\-ery\-whe\-re (q.e.)\ on $\mathbb R^n$, cf.\
\cite[Section~III.1.1]{L}, and hence also nearly everywhere (n.e.)\ on $\mathbb R^n$.\footnote{An assertion $\mathcal A(x)$ involving a variable point $x\in\mathbb R^n$ is said to hold {\it qua\-si-ev\-ery\-whe\-re} (resp.\ {\it nearly everywhere}) on $A\subset\mathbb R^n$ if the set $E$ of all $x\in A$ where $\mathcal A(x)$ fails, is of outer (resp.\ inner) capacity zero. Regarding the {\it outer} and {\it inner} capacities, denoted by $c^*(\cdot)$ and $c_*(\cdot)$, respectively, see \cite[Section~II.2.6]{L}. We write $c(E):=c_*(E)$ if $c_*(E)=c^*(E)$ (e.g.\ if $E$ is Borel \cite[Theorem~2.8]{L}).}

A crucial fact discovered by Riesz \cite[Chapter~I, Eq.~(13)]{R} (cf.\ \cite[Theorem~1.15]{L}) is that the kernel $\kappa_\alpha$ is {\it strictly positive definite}, that is, $I(\mu)\geqslant0$ for every (signed) $\mu\in\mathfrak M$, and moreover $I(\mu)=0\iff\mu=0$. This implies that all (signed) $\mu\in\mathfrak M$ with $I(\mu)<\infty$ form a pre-Hilbert space $\mathcal E$ with the inner product $\langle\mu,\nu\rangle:=I(\mu,\nu)$ and the energy norm $\|\mu\|:=\sqrt{I(\mu)}$, see \cite[Lemma~3.1.2]{F1}. The topology on $\mathcal E$ defined by means of the norm $\|\cdot\|$ is said to be {\it strong}.

Moreover, due to Deny \cite{D1} (for $\alpha=2$, cf.\ also Cartan \cite{Ca1}), the cone $\mathcal E^+:=\mathcal E\cap\mathfrak M^+$ is {\it complete} in the induced strong topology, and the strong topology on $\mathcal E^+$ is {\it finer} than the induced vague topology on $\mathcal E^+$.\footnote{A kernel having these two properties is said to be {\it perfect} (Fuglede \cite[Section~3.3]{F1}).} Thus any strong Cauchy sequence (net) $(\mu_j)\subset\mathcal E^+$ converges both strongly and vaguely to the same (unique) limit $\mu_0\in\mathcal E^+$, the strong topology on $\mathcal E$ as well as the vague topology on $\mathfrak M$ being Hausdorff.\footnote{It is worth noting that the whole pre-Hilbert space $\mathcal E$ is in general strongly incomplete (see a counterexample by Cartan \cite{Ca1}, pertaining to the Newtonian kernel; cf.\ \cite[Theorem~1.19]{L}).}

For any $A\subset\mathbb R^n$, let $\mathfrak M^+(A)$ stand for the class of all $\mu\in\mathfrak M^+$ {\it concentrated on} $A$, which means that $A^c:=\mathbb R^n\setminus A$ is $\mu$-neg\-lig\-ible, or equivalently that the set $A$ is $\mu$-meas\-ur\-able and $\mu=\mu|_A$, $\mu|_A$ being the trace of $\mu$ to $A$, cf.\ \cite[Section~V.5.7]{B2}. If $A$ is closed, then $\mu\in\mathfrak M^+$ is concentrated on $A$ if and only if $S(\mu)\subset A$, $S(\mu)$ being the support of $\mu$.
We also denote
\[\mathcal E^+(A):=\mathfrak M^+(A)\cap\mathcal E,\quad\breve{\mathcal E}^+(A):=\bigl\{\mu\in\mathcal E^+(A):\ \mu(\mathbb R^n)=1\bigr\}.\]

To avoid trivialities, we shall always assume that
\begin{equation}\label{nonzero}
 c_*(A)>0;
\end{equation}
then (and only then) $\mathcal E^+(A)\ne\{0\}$, cf.\ \cite[Lemma~2.3.1]{F1}, and hence $\breve{\mathcal E}^+(A)\ne\varnothing$.

For the Gauss functional $I_{f_{q,z}}(\mu)$, introduced by (\ref{ext2}) or equivalently by
\begin{equation}\label{G}
I_{f_{q,z}}(\mu)=I(\mu)-2q\int U^{\varepsilon_z}\,d\mu=I(\mu)-2qU^\mu(z),
\end{equation}
$\varepsilon_z$ being the unit Dirac measure at $z\in\mathbb R^n$, we define
\begin{equation*}
w_{f_{q,z}}(A):=\inf_{\mu\in\breve{\mathcal E}^+(A)}\,I_{f_{q,z}}(\mu).
\end{equation*}
Clearly,
\[-\infty\leqslant w_{f_{q,z}}(A)<\infty.\]

\begin{problem}\label{pr}
If $w_{f_{q,z}}(A)>-\infty$, does there exist $\lambda_{A,f_{q,z}}\in\breve{\mathcal E}^+(A)$ with
\[I_{f_{q,z}}(\lambda_{A,f_{q,z}})=w_{f_{q,z}}(A)?\]
\end{problem}

Problem~\ref{pr}, being initiated by Gauss \cite{Gau} (see also Frostman \cite{Fr} and Ohtsuka \cite{O}), is still of interest due to its important applications in various areas of mathematics (see the monographs \cite{BHS,ST} and references therein; cf.\ \cite{CSW,Dr0} and \cite{Z-Rarx}--\cite{Z-pseudo} for some of the latest researches on this topic).

The solution $\lambda_{A,f_{q,z}}$ to Problem~\ref{pr} is {\it unique}, which follows easily from the convexity of the class $\breve{\mathcal E}^+(A)$ by making use of the parallelogram identity in the
pre-Hil\-bert space $\mathcal E$ and the strict positive definiteness of the Riesz kernel (cf.\ \cite[Lemma~6]{Z5a}).

The question on the existence of the solution $\lambda_{A,f_{q,z}}$ is, however, not trivial~--- not even if the set $A:=K$ is compact (despite the vague compactness of $\breve{\mathcal E}^+(K)$, see \cite[Section~III.1.9, Corollary~3]{B2}). This is caused by the fact that the Gauss functional $I_{f_{q,z}}(\mu)$, $\mu\in\mathcal E^+$, is {\it not} vaguely lower semicontinuous (l.s.c.), being the difference of two vaguely l.s.c.\ functions. (Compare e.g.\ with \cite[Proof of Theorem~2.6]{O}.)

In this work, we obtain necessary and sufficient conditions for the existence of the minimizer $\lambda_{A,f_{q,z}}$, given in terms of the inner $\alpha$-harmonic measure $\varepsilon_z^A$ (Theorem~\ref{th1}).\footnote{Regarding {\it the inner $\alpha$-harmonic measure} $\varepsilon_z^A:=(\varepsilon_z)^A$ of any $A\subset\mathbb R^n$, see \cite[Section~2]{Z-bal2} (cf.\ also Section~\ref{sec-harm} below), $\mu^A$ being the inner ($\alpha$-Riesz) balayage of $\mu\in\mathfrak M^+$ to $A$. If $A$ is  representable as a countable union of compact sets (thus in particular if $A$ is closed or open), then $\varepsilon_z^A$ coincides with the {\it outer} $\alpha$-har\-mon\-ic measure $\overline{\varepsilon}_z^A$ (see \cite[Corollary~6.1]{Z-bal2}), investigated e.g.\ in \cite{BH}--\cite{KB'}, \cite{L}.} In more details, for any $A$ having the property $(\mathcal P)$ (Section~\ref{sec-main}), and any $z\in A^u\cup{\overline A}^c$, we establish a complete description of the set $\mathfrak S(A,z)$ of all $q\in(0,\infty)$ such that $\lambda_{A,f_{q,z}}$ does exist. Here $\overline A:={\rm Cl}_{\mathbb R^n}A$, whereas $A^u$ is the set of all inner $\alpha$-{\it ul\-tra\-ir\-reg\-ular} points for $A$, namely of all those $x\in\partial A$ $\bigl({}:=\partial_{\mathbb R^n}A\bigr)$ for which $\varepsilon_x^A$ is of finite energy.\footnote{The concept of inner $\alpha$-ultrairregular point, deeply related to that of inner $\alpha$-harmonic measure, is newly introduced. See Section~\ref{sec-harm} for more details, relevant facts, and an illustrative example.}

Also, we establish alternative characterizations of this $\lambda_{A,f_{q,z}}$ (Theorem~\ref{th-alt}), and describe its support (Theorem~\ref{th2}). While treating $\lambda_{A,f_{q,z}}$ as a function of $z\in\overline A^c$, we prove its vague continuity,\footnote{For the terminology used here we refer to Bourbaki \cite[Section~V.3.1]{B2}.} and find out accumulations points (Theorem~\ref{th3}).

The results thereby obtained discover new interesting phenomena, some of which seem to contradict our physical intuition; see Corollaries~\ref{cor0}, \ref{cor1}, \ref{cor2} and Remarks~\ref{rem1}, \ref{rem2}, \ref{rem3}. Examples~\ref{ex} and \ref{ex4}--\ref{ex3} are provided for the sake of illustration.

\section{Main results}\label{sec-main}

Henceforth, we maintain the conventions adopted in Section~\ref{sec-alpha}. Furthermore,
unless explicitly stated otherwise, $A\subset\mathbb R^n$, $\overline{A}\ne\mathbb R^n$, is a fixed set having the property\footnote{As shown in \cite[Theorem~3.9]{Z-Rarx}, $(\mathcal P)$ is fulfilled e.g.\ if $A$ is {\it quasiclosed} ({\it quasicompact}), that is, if $A$ can be approximated in outer capacity by closed (compact) sets (Fuglede \cite[Definition~2.1]{F71}).}
\begin{itemize}
  \item[$(\mathcal P)$] The cone $\mathcal E^+(A)$ is closed in the strong topology on $\mathcal E^+$.
  \end{itemize}

Let $H_{A,x}$ be the inverse of the total mass of the inner $\alpha$-harmonic measure $\varepsilon_x^A$:
\begin{equation}\label{H}
 H_x:=H_{A,x}:=1/\varepsilon_x^A(\mathbb R^n),\quad x\in\mathbb R^n.
\end{equation}
Then
\[1\leqslant H_x<\infty,\]
the former inequality being obvious from $\varepsilon_x^A(\mathbb R^n)\leqslant\varepsilon_x(\mathbb R^n)$, and the latter from (\ref{nonzero}) (see Section~\ref{some}, (b) and (d), respectively).

As seen from Theorem~\ref{harm-tot} below, the same $H_x$ can equivalently be defined by
\begin{equation}\label{n4}H_x=\left\{
\begin{array}{cl}1&\text{if \ $A$ is not inner $\alpha$-thin at infinity},\\
1/U^{\gamma_A}(x)&\text{otherwise},\\ \end{array} \right.
\end{equation}
where $\gamma_A$ is the inner equilibrium measure for $A$, treated in an extended sense where both $\gamma_A(\mathbb R^n)$ and $I(\gamma_A)$ may be $+\infty$. For the concept of inner $\alpha$-thinness of a set at infinity and that of inner equilibrium measure, see \cite{Z-bal2} (cf.\ also Section~\ref{sec-eq} below).

\subsection{On the solvability of Problem~\ref{pr}} The following theorem provides necessary and sufficient conditions for the existence of the solution $\lambda_{A,f_{q,z}}$ to Problem~\ref{pr}.

\begin{theorem}\label{th1} For any given $z\in A^u\cup{\overline A}^c$,
$\lambda_{A,f_{q,z}}$ exists if and only if either
\begin{equation}\label{capfin}
 c_*(A)<\infty,\end{equation}
 or
\begin{equation}\label{q}
q\geqslant H_z.
\end{equation}
\end{theorem}

$\P$ Given $z\in A^u\cup{\overline A}^c$, the set $\mathfrak S(A,z)$ of all $q\in(0,\infty)$ such that $\lambda_{A,f_{q,z}}$ does exist is thus shown to be equal to $(0,\infty)$ if $c_*(A)<\infty$, and to $[H_z,\infty)$ otherwise.

\begin{remark}\label{rem1} Assume for a moment that $A$ is closed, $z\in A^u$, and (\ref{capfin}) or (\ref{q}) is met. Then, according to Theorem~\ref{th1}, $w_{f_{q,z}}(A)$ is {\it an actual minimum}, and hence no compensation effect occurs between the two oppositely signed charges, $-q\varepsilon_z$ and $\lambda_{A,f_{q,z}}$, carried by the same conductor $A$. (Observe that this still holds even when $q$, the attractive mass placed at $z$, becomes arbitrarily large.) The phenomenon thereby discovered seems to be quite surprising and even to contradict our physical intuition. Therefore, it would be interesting to illustrate this phenomenon by means of numerical experiments, which must of course depend on the geometry of the conductor $A$ in a neighborhood of the $\alpha$-ul\-t\-ra\-irregular point $z$, cf.\ Examples~\ref{ex}, \ref{ex1}--\ref{ex3}.
\end{remark}

\begin{corollary}\label{cor0}Assume $A$ is quasicompact. Then for any $z\in A^u\cup{\overline A}^c$ and any $q\in(0,\infty)$,
$\lambda_{A,f_{q,z}}$ does exist.\end{corollary}

\begin{proof}
  This follows directly from Theorem~\ref{th1}, for, being approximated in outer capacity by compact sets, $A$ has finite outer (hence, inner) capacity.
\end{proof}

In the following Corollaries~\ref{cor1} and \ref{cor2}, we explore the solvability of Problem~\ref{pr} for $z$ lying in an arbitrarily small neighborhood of $\overline{A}$ in the topology of $\mathbb R^n\cup\{\infty_{\mathbb R^n}\}$, the one-point compactification of $\mathbb R^n$.

\begin{corollary}\label{cor1} Assume $A$ is closed, $\alpha$-thin at infinity, and of infinite capacity.\footnote{Such $A$ does exist, see \cite{Z-bal2} (Remark~2.2 and Example~2.1), cf.\ also Example~\ref{ex2} below.} Then there exist two sequences, $(z_j)\subset A^c$ and $(q_j)\in[1,\infty)$, having the properties
\[\lim_{j\to\infty}\,|z_j|=\lim_{j\to\infty}\,q_j=\infty\]
and such that
\begin{equation*}
 \lambda_{A,f_{q,z_j}}\text{\ exists}\iff q\geqslant q_j.
\end{equation*}
\end{corollary}

\begin{remark}\label{rem0}Corollary~\ref{cor1} is sharp in the sense that it would no longer be valid if the assumption $c(A)=\infty$ or that of $\alpha$-thinness of $A$ at infinity were dropped from its hypotheses. Indeed, this is obvious from Theorem~\ref{th1} on account of (\ref{n4}).\end{remark}

\begin{remark}\label{rem0'}The sequence $(q_j)$, appearing in Corollary~\ref{cor1}, is uniquely determined by the sequence $(z_j)$; we actually have $q_j:=H_{z_j}$ (see Section~\ref{pr-cor1} for a proof). A similar remark is also applicable to Corollary~\ref{cor2}, cf.\ Section~\ref{pr-cor2}.\end{remark}

Let $A^r$ stand for the set of all {\it inner $\alpha$-reg\-ul\-ar points} of $A$ (see \cite[Section~6]{Z-bal} or Section~\ref{sec-reg} below; compare with \cite[Section~V.1.2]{L}, where $A$ is Borel).

\begin{corollary}\label{cor2}Assume that $c_*(A)=\infty$. Then for each $y\in A^r\cap\partial A$, there exist a sequence $(z_j)\subset\overline{A}^c$ approaching $y$, and a sequence $(q_j)\in[1,\infty)$ approaching~$1$, such that
\begin{equation*}
 \lambda_{A,f_{q,z_j}}\text{\ exists}\iff q\geqslant q_j.
\end{equation*}
\end{corollary}

\begin{remark}\label{rem4} Corollary~\ref{cor2} holds trivially if $A$ is not inner $\alpha$-thin at infinity. Actually, then we have even more~--- namely, for any given $z\in\overline{A}^c$,
\begin{equation*}
 \lambda_{A,f_{q,z}}\text{\ exists}\iff q\geqslant1
 \end{equation*}
(see the former relation in (\ref{n4}) and Theorem~\ref{th1}).\end{remark}

\begin{remark}\label{rem5} Corollary~\ref{cor2} would fail to hold if $c_*(A)$ were assumed to be finite, see Theorem~\ref{th1}.\end{remark}

\subsection{Alternative characterizations of $\lambda_{A,f_{q,z}}$}\label{sec-alt}  Fix $z\in A^u\cup{\overline A}^c$. Throughout Sections~\ref{sec-alt}--\ref{sec-conv}, we assume that either of (\ref{capfin}) or (\ref{q}) is fulfilled; then (and only then) the solution $\lambda_{A,f_{q,z}}$ to Problem~\ref{pr} does exist (Theorem~\ref{th1}).

By Lemma~\ref{l-char}, $\lambda_{A,f_{q,z}}$ is uniquely characterized within $\breve{\mathcal E}^+(A)$ by the inequality\footnote{This implies that $c_{A,f_{q,z}}$ is finite, $U^{\lambda_{A,f_{q,z}}}_{f_{q,z}}$ being finite q.e.\ (hence n.e.)\ on $\mathbb R^n$, see Section~\ref{sec-alpha}. Alternatively, $c_{A,f_{q,z}}$ is finite, for so are both $I_{f_{q,z}}(\lambda_{A,f_{q,z}})$ and $I(\lambda_{A,f_{q,z}})$.}
\begin{equation}\label{e-char}
U^{\lambda_{A,f_{q,z}}}_{f_{q,z}}\geqslant c_{A,f_{q,z}}\quad\text{\ n.e. on $A$},
\end{equation}
where
\begin{equation*}U^{\lambda_{A,f_{q,z}}}_{f_{q,z}}:=U^{\lambda_{A,f_{q,z}}}+f_{q,z}
\end{equation*}
is said to {\it the $f_{q,z}$-weighted potential} of the measure $\lambda_{A,f_{q,z}}$, while
\begin{equation}\label{iec}
 c_{A,f_{q,z}}:=\int U^{\lambda_{A,f_{q,z}}}_{f_{q,z}}\,d\lambda_{A,f_{q,z}}
 \quad\Bigl({}=w_{f_{q,z}}(A)+qU^{\lambda_{A,f_{q,z}}}(z)\Bigr)
\end{equation}
is termed {\it the inner $f_{q,z}$-weighted equilibrium constant} for the set $A$.

Denote
\begin{equation}\label{L}
 \Lambda_{A,f_{q,z}}:=\bigl\{\mu\in\mathfrak M^+:\ U^\mu_{f_{q,z}}\geqslant c_{A,f_{q,z}}\quad\text{n.e. on $A$}\bigr\}.
\end{equation}
As seen from (\ref{e-char}),
\[\lambda_{A,f_{q,z}}\in\Lambda_{A,f_{q,z}}.\]

\begin{theorem}\label{th-alt}In addition to the hypotheses indicated at the beginning of this subsection, assume that
\begin{equation*}
q\leqslant H_z.
\end{equation*}
Then $\lambda_{A,f_{q,z}}$ is representable in the form
\begin{equation}\label{sol}\lambda_{A,f_{q,z}}=\left\{
\begin{array}{cl}q\varepsilon_z^A+c_{A,f_{q,z}}\gamma_A&\text{if \ $q<H_z$},\\
H_z\varepsilon_z^A&\text{otherwise},\\ \end{array} \right.
\end{equation}
and it can alternatively be characterized by any one of the following {\rm(i)--(iii)}.
\begin{itemize}
  \item[{\rm(i)}] $\lambda_{A,f_{q,z}}$ is the unique solution to the problem of minimizing the potential $U^\mu$ over the class $\Lambda_{A,f_{q,z}}$. That is,\footnote{Relations (\ref{alt2}) and (\ref{alt3}) would certainly be the same if $\mu$ in (\ref{L}) were of finite energy.}
      \begin{equation}\label{alt2}
       U^{\lambda_{A,f_{q,z}}}=\min_{\mu\in\Lambda_{A,f_{q,z}}}\,U^\mu\quad\text{on $\mathbb R^n$}.
      \end{equation}
  \item[{\rm(ii)}] $\lambda_{A,f_{q,z}}$ is the unique solution to the problem of minimizing the energy $I(\mu)$ over the class $\Lambda_{A,f_{q,z}}$. That is,
      \begin{equation}\label{alt3}
       I(\lambda_{A,f_{q,z}})=\min_{\mu\in\Lambda_{A,f_{q,z}}}\,I(\mu).
      \end{equation}
  \item[{\rm(iii)}] $\lambda_{A,f_{q,z}}$ is the only measure in the class $\mathcal E^+(A)$ having the property
  \begin{equation*}
  U^{\lambda_{A,f_{q,z}}}_{f_{q,z}}=c_{A,f_{q,z}}\quad\text{n.e.\ on $A$}.
  \end{equation*}
Furthermore,\footnote{Relation (\ref{R'}) can actually be specified as follows (see (\ref{elat7})):
\[c_{A,f_{q,z}}=\frac{H_z-q}{H_zc_*(A)}\quad\text{if $q<H_z$}.\]
Note that $q<H_z$ necessarily yields $c_*(A)<\infty$, for if not, $\lambda_{A,f_{q,z}}$ fails to exist (Theorem~\ref{th1}).}
\begin{align}\label{R'}
 c_{A,f_{q,z}}&>0\quad\text{if $q<H_z$},\\
 c_{A,f_{q,z}}&=0\quad\text{otherwise}.\label{RR'}
\end{align}
\end{itemize}
\end{theorem}

\begin{remark}As shown below (see (\ref{i0})), for any $A$ and any $z\in A^u\cup{\overline A}^c$, we have, along with (\ref{R'}) and (\ref{RR'}),
\begin{equation}\label{RRR'}
c_{A,f_{q,z}}<0\quad\text{if $q>H_z$}.
\end{equation}
\end{remark}

\subsection{Is $S(\lambda_{A,f_{q,z}})$ compact or noncompact?}\label{sec-comp}
As seen from Theorem~\ref{th2} (cf.\ also Remark~\ref{rem2}), an answer to this question strongly depends on $q\in\mathfrak S(A,z)$.

The {\it reduced kernel} of $A$ is the set of all $x\in A$ such that for every open ball $B_{x,r}$ of radius $r\in(0,\infty)$ centered at $x$, we have $c_*(A\cap B_{x,r})>0$, cf.\ \cite[p.~164]{L}.

\begin{theorem}\label{th2} For any $z\in A^u\cup{\overline A}^c$, the following {\rm(i$_1$)} and {\rm(ii$_1$)} hold true.
\begin{itemize}
\item[{\rm(i$_1$)}] For any $q>H_z$, $\lambda_{A,f_{q,z}}$ is of compact support.
\item[{\rm(ii$_1$)}] Assume that $A$ is closed, coincides with its reduced kernel, and either $\alpha<2$ or $A^c$ is connected. If moreover $q\leqslant H_z$,
    then
    \begin{equation}\label{R}S(\lambda_{A,f_{q,z}})=\left\{
\begin{array}{cl}A&\text{if \ $\alpha<2$},\\
\partial A&\text{otherwise}.\\ \end{array} \right.
\end{equation}
\end{itemize}
\end{theorem}

\begin{remark}\label{rem2} Under the hypotheses of Theorem~\ref{th2}(ii$_1$), assume that $\partial A$ is unbounded. Then, by virtue of (\ref{R}), the support of $\lambda_{A,f_{H_z,z}}$ is noncompact, whereas that of
$\lambda_{A,f_{H_z+\beta,z}}$ is already compact for any $\beta\in(0,\infty)$, see Theorem~\ref{th2}(i$_1$). Thus, any increase (even arbitrarily small) of the attractive mass $H_z$ placed at $z\in A^u\cup A^c$, changes drastically the solution to the problem in question, making it zero in some neighborhood of the Alexandroff point $\infty_{\mathbb R^n}$.\end{remark}

\begin{remark}
Theorem~\ref{th2} gives an answer to Open question~2.1 raised by Ohtsuka \cite[p.~284]{O} about conditions ensuring the identity $S(\lambda_{A,f_{q,z}})=A$.
\end{remark}

\subsection{On the vague convergence of $\lambda_{A,f_{q,z}}$}\label{sec-conv} The following theorem analyzes the vague convergence of $\lambda_{A,f_{q,z}}$, treated as a function of $z\in\overline{A}^c$.

\begin{theorem}\label{th3} Given $A\subset\mathbb R^n$, the following {\rm(i$_2$)} and {\rm(ii$_2$)} hold true.
\begin{itemize}
\item[{\rm(i$_2$)}] For any $y\in A^r\cap\partial A$,
\begin{equation}\label{vcon}\lim_{\overline{A}^c\ni z\to y}\,\lambda_{A,f_{H_z,z}}(\varphi)=\varepsilon^A_y(\varphi)=\varepsilon_y(\varphi)\quad\text{for all $\varphi\in C_0(\mathbb R^n)$}.
\end{equation}
\item[{\rm(ii$_2$)}] Assume that either $q=H_z$, or $q$ is independent of $z$, and moreover $q\leqslant1$.\footnote{In either of these two cases, Theorem~\ref{th-alt} is applicable, and hence $\lambda_{A,f_{q,z}}$ is given by (\ref{sol}).} Then the mapping $z\mapsto\lambda_{A,f_{q,z}}$, $z\in\overline{A}^c$, is vaguely continuous, that is,
\begin{equation*}
\lim_{\overline{A}^c\ni z\to y\in\overline{A}^c}\,\lambda_{A,f_{q,z}}(\varphi)=\lambda_{A,f_{q,y}}(\varphi)\quad\text{for all $\varphi\in C_0(\mathbb R^n)$.}
\end{equation*}
\end{itemize}
\end{theorem}

\begin{remark}\label{rem3}According to (\ref{vcon}), $y\in A^r\cap\partial A$ serve as accumulation points for $\lambda_{A,f_{H_z,z}}$ when $z$ ranges over $\overline{A}^c$, that is,
\begin{equation*}
\lambda_{A,f_{H_z,z}}\to\varepsilon_y\quad\text{vaguely as $\overline{A}^c\ni z\to y\in A^r\cap\partial A$}.\end{equation*}
This phenomenon seems to be unexpected in view of the fact that the support of every $\lambda_{A,f_{H_z,z}}$, $z\in\overline{A}^c$, might even coincide with the whole $A$, see Theorem~\ref{th2}(ii$_1$).\end{remark}

\subsection{On the novelty of the above results} The results of this section are largely new. However, in the particular case where $A$ is closed and not $\alpha$-thin at infinity, while $z\in A^c$, necessary and sufficient conditions for the existence of $\lambda_{A,f_{q,z}}$ as well as sufficient conditions for the compactness of its support (given above by (\ref{q}) and Theorem~\ref{th2}(i$_1$), respectively, with $H_z:=1$) were established by Dragnev et al.\ \cite[Corollary~2.6]{Dr0}. Our approach is substantially different from that in \cite{Dr0}, which enabled us to investigate the problem in question for
{\it any} quasiclosed $A$ (independently of how it is large at infinity) and {\it any} $z\in\overline{A}^c\cup A^u$. Such an improvement is in particular due to the systematic use of the strong topology on $\mathcal E$, which in turn became possible since the inner $\alpha$-harmonic measure $\varepsilon_z^A$, $z\in\overline{A}^c\cup A^u$, is of finite energy.\footnote{This observation suggests that some of the results of the present study can be extended to the external field $f:=-U^\mu$, where $\mu\in\mathfrak M^+$ is an arbitrary measure having the property $\mu^A\in\mathcal E^+$. Such a generalization might be a subject of a further research work.}

Proofs are given in Sections~\ref{sec-prep} and \ref{sec-proofs}.
The (new) concept of inner $\alpha$-ul\-t\-ra\-ir\-reg\-ul\-ar point is introduced in Section~\ref{sec-harm}. A background for this definition is provided in Sections~\ref{some}--\ref{sec-eq}, reviewing some basic facts of the theories of inner balayage and inner equilibrium measures, the summary being organized correspondingly to its usage in the current research. Finally, in Section~\ref{sec-appl} we present some illustrative examples.

\section{Inner $\alpha$-harmonic measure and inner $\alpha$-Riesz equilibrium measure}\label{sec-baleq}

Unless explicitly stated otherwise, in this section we do not require $(\mathcal P)$ to hold.

\subsection{Basic facts about inner balayage}\label{some} The theory of inner ($\alpha$-Riesz) balayage of arbitrary $\zeta\in\mathfrak M^+$ to arbitrary $A\subset\mathbb R^n$, $n\geqslant2$, generalizing Cartan's theory \cite{Ca2} of inner Newtonian balayage ($\alpha=2$, $n\geqslant3$) to any $\alpha\in(0,2]$, $\alpha<n$, was originated in the author's recent papers \cite{Z-bal,Z-bal2}, and it found a further development in \cite{Z-arx1}--\cite{Z-arx-22}.\footnote{For the theory of {\it outer} Riesz balayage, see e.g.\ the monographs \cite{BH,Br,Doob}, the last two dealing with the Newtonian kernel ($\alpha=2$, $n\geqslant3$).}

The concept of inner balayage $\zeta^A$ of any $\zeta\in\mathfrak M^+$ to any $A\subset\mathbb R^n$ is introduced by means of the following Theorems~\ref{th-bal1} and \ref{th-bal2}. We denote
\begin{equation}\label{gammaa}
   \Gamma_{A,\zeta}:=\bigl\{\mu\in\mathfrak M^+:\ U^\mu\geqslant U^\zeta\quad\text{n.e.\ on $A$}\bigr\},
  \end{equation}
and let $\mathcal E'(A)$ be the closure of $\mathcal E^+(A)$ in the strong topology on $\mathcal E^+$. Being a strongly closed subcone of the strongly complete cone $\mathcal E^+$ (Section~\ref{sec-alpha}), $\mathcal E'(A)$ is likewise strongly complete, and it is convex.

\begin{theorem}\label{th-bal1}
For any $\zeta:=\sigma\in\mathcal E^+$, there is precisely one $\sigma^A\in\mathcal E'(A)$, called the inner balayage of $\sigma$ to $A$, that is determined by any one of the following {\rm(i$_3$)}--{\rm(iii$_3$)}.
\begin{itemize}
  \item[{\rm(i$_3$)}] There exists the unique $\sigma^A\in\mathcal E'(A)$ having the property\footnote{That is, the inner balayage $\sigma^A$ of $\sigma\in\mathcal E^+$ to $A$ is, actually, the orthogonal projection of $\sigma$ in the pre-Hil\-bert space $\mathcal E$ onto the (convex, strongly complete) cone $\mathcal E'(A)$, cf.\ \cite[Theorem~1.12.3]{E2}.}
  \[\|\sigma-\sigma^A\|=\min_{\mu\in\mathcal E'(A)}\,\|\sigma-\mu\|.\]
  \item[{\rm(ii$_3$)}] There exists the unique $\sigma^A\in\mathcal E'(A)$ satisfying the equality
  \begin{equation*}
   U^{\sigma^A}=U^\sigma\quad\text{n.e.\ on $A$}.
  \end{equation*}
  \item[{\rm(iii$_3$)}] $\sigma^A$ is the unique solution to the problem of minimizing the energy over the class $\Gamma_{A,\sigma}$, introduced by {\rm(\ref{gammaa})} with $\zeta:=\sigma$. That is, $\sigma^A\in\Gamma_{A,\sigma}$  and
  \[I(\sigma^A)=\min_{\mu\in\Gamma_{A,\sigma}}\,I(\mu).\]
 \end{itemize}
\end{theorem}

\begin{proof}
  See \cite[Section~3]{Z-bal}, \cite[Section~4]{Z-arx1}, and \cite[Section~3]{Z-arx-22}.
\end{proof}

\begin{remark}\label{r-h} If $A$ satisfies $(\mathcal P)$ (in particular, if $A$ is quasiclosed), then
\begin{equation}\label{ee}
 \mathcal E'(A)=\mathcal E^+(A),
\end{equation}
and hence (i$_3$) and (ii$_3$) in Theorem~\ref{th-bal1} remain valid with $\mathcal E^+(A)$ in place of $\mathcal E'(A)$. The inner balayage $\sigma^A$, where $\sigma\in\mathcal E^+$, is then necessarily concentrated on $A$, and it is uniquely determined within $\mathcal E^+(A)$ by the equality $U^{\sigma^A}=U^\sigma$ n.e.\ on $A$.\footnote{For arbitrary $A$, this might fail to hold, which can be seen by means of elementary examples.}
\end{remark}

\begin{theorem}\label{th-bal2}For any $\zeta\in\mathfrak M^+$, there exists precisely one $\zeta^A\in\mathfrak M^+$, called the inner balayage of $\zeta$ to $A$, that is determined by any one of the following {\rm(i$_4$)}--{\rm(iii$_4$)}.
\begin{itemize}
\item[{\rm(i$_4$)}] $\zeta^A$ is the unique solution to the problem of minimizing the potential $U^\mu$ over the class $\Gamma_{A,\zeta}$, that is, $\zeta^A\in\Gamma_{A,\zeta}$ and
  \[U^{\zeta^A}=\min_{\mu\in\Gamma_{A,\zeta}}\,U^\mu\quad\text{on $\mathbb R^n$}.\]
  \item[{\rm(ii$_4$)}] There exists the unique $\zeta^A\in\mathfrak M^+$ satisfying the symmetry relation
  \begin{equation}\label{eq-sym}
    I(\zeta^A,\sigma)=I(\zeta,\sigma^A)\quad\text{for all $\sigma\in\mathcal E^+$},
  \end{equation}
  where $\sigma^A$ is uniquely determined by Theorem~{\rm\ref{th-bal1}} {\rm(cf.\ also Remark~\ref{r-h})}.
   \item[{\rm(iii$_4$)}] There exists the unique $\zeta^A\in\mathfrak M^+$ satisfying either of the two limit relations
  \begin{align*}\sigma_j^A&\to\zeta^A\quad\text{vaguely in $\mathfrak M^+$ as $j\to\infty$},\\
U^{\sigma_j^A}&\uparrow U^{\zeta^A}\quad\text{pointwise on $\mathbb R^n$ as $j\to\infty$},
\end{align*}
where $(\sigma_j)\subset\mathcal E^+$ denotes an arbitrary sequence having the property\footnote{Such a sequence $(\sigma_j)\subset\mathcal E^+$ does exist (see e.g.\ \cite[p.~272]{L} or \cite[p.~257, footnote]{Ca2}).}
\begin{equation}\label{eq-mon}U^{\sigma_j}\uparrow U^\zeta\quad\text{pointwise on $\mathbb R^n$ as
$j\to\infty$},\end{equation}
while $\sigma_j^A$ is uniquely determined by Theorem~{\rm\ref{th-bal1}} {\rm(cf.\ also Remark~\ref{r-h})}.
\end{itemize}
\end{theorem}

\begin{proof}
  See \cite[Sections~3, 4]{Z-bal}.
\end{proof}

\begin{remark}\label{rem-nodet}
  It is worth noting that, although for $\zeta\in\mathfrak M^+$, we still have
  \begin{equation}\label{ineq1}
   U^{\zeta^A}=U^\zeta\quad\text{n.e.\ on $A$},
  \end{equation}
  see \cite[Theorem~3.10]{Z-bal}, this equality no longer determines $\zeta^A$ uniquely (as it does for $\zeta:=\sigma\in\mathcal E^+$, cf.\ Theorem~\ref{th-bal1}(ii$_3$) or Remark~\ref{r-h}), which can be seen by taking $\zeta:=\varepsilon_y$, $y$ being an inner $\alpha$-irregular point of $A$ (for definition see Section~\ref{sec-reg}).
\end{remark}

\begin{remark}
Relation (\ref{eq-sym}) can be extended to $\sigma$ of infinite energy, that is,
\begin{equation}\label{alt}
I(\zeta^A,\mu)=I(\zeta,\mu^A)\quad\text{for all $\zeta,\mu\in\mathfrak M^+$}.
\end{equation}
Indeed, choose a sequence $(\sigma_j)\subset\mathcal E^+$ satisfying (\ref{eq-mon}). By virtue of Theorem~\ref{th-bal2}, see (ii$_4$) and (iii$_4$), the sequence $\bigl(U^{\sigma_j^A}\bigr)$ increases pointwise on $\mathbb R^n$ to $U^{\zeta^A}$, whereas
 \[\int U^{\sigma_j^A}\,d\mu=\int U^{\sigma_j}\,d\mu^A\quad\text{for all\ $j\in\mathbb N$}.\]
 Applying the monotone convergence theorem \cite[Section~IV.1, Theorem~3]{B2} to each of these two integrals, we obtain (\ref{alt}).
\end{remark}

Given $A\subset\mathbb R^n$, we denote by $\mathfrak C_A$ the upward directed set of all compact subsets $K$ of $A$, where $K_1\leqslant K_2$ if and only if $K_1\subset K_2$. If a net $(x_K)_{K\in\mathfrak C_A}\subset Y$ converges to $x_0\in Y$, $Y$ being a topological space, then we shall indicate this fact by writing
\begin{equation*}x_K\to x_0\text{ \ in $Y$ as $K\uparrow A$}.\end{equation*}

Along with (\ref{ineq1}), the following properties of the inner balayage $\zeta^A$, $\zeta\in\mathfrak M^+$ and $A\subset\mathbb R^n$ being arbitrary, will be useful in the sequel.

\begin{itemize}
\item[(a)] $U^{\zeta^A}\leqslant U^\zeta$ everywhere on $\mathbb R^n$ (see \cite[Theorem~3.10]{Z-bal}).
\item[(b)] Principle of positivity of mass: $\zeta^A(\mathbb R^n)\leqslant\zeta(\mathbb R^n)$ (see \cite[Corollary~4.9]{Z-bal}).
\item[(c)] Balayage "with a rest": $\zeta^A=(\zeta^Q)^A$ for any $Q\supset A$ (see \cite[Corollary~4.2]{Z-bal}).
\item[(d)] If $c_*(A)>0$, then $\zeta^A=0\iff\zeta=0$.
\item[(e)] Convergence assertions (see \cite[Theorem~4.5]{Z-bal}):
\begin{align}\label{conv1}
 \zeta^K&\to\zeta^A\quad\text{vaguely in $\mathfrak M^+$ as $K\uparrow A$},\\
 U^{\zeta^K}&\uparrow U^{\zeta^A}\quad\text{pointwise on $\mathbb R^n$ as $K\uparrow A$}.\notag
\end{align}
If $\zeta\in\mathcal E^+$, then the net $(\zeta^K)_{K\in\mathfrak C_A}$ also converges to $\zeta^A$ strongly in $\mathcal E^+$.
\end{itemize}

In the rest of this subsection, $A$ is assumed to satisfy $(\mathcal P)$. Then Theorem~\ref{th-bal1}(i$_3$) admits the following useful generalization.

\begin{theorem}\label{l-oo'}
Given $\zeta\in\mathfrak M^+$, assume that the inner balayage $\zeta^A$ is of finite energy. Then $\zeta^A$ is concentrated on $A$, i.e.\
$\zeta^A\in\mathcal E^+(A)$, and it can actually be found as the unique solution to the problem
\[\min_{\mu\in\mathcal E^+(A)}\,\Bigl[\|\mu\|^2-2\int U^\zeta\,d\mu\Bigr].\]
The same $\zeta^A$ is the only measure in $\mathcal E^+(A)$ with the property $U^{\zeta^A}=U^\zeta$ n.e.\ on $A$.
\end{theorem}

\begin{proof} Since $\zeta^A=(\zeta^A)^A$ (see (c) with $Q:=A$), substituting (\ref{ee}) into Theorem~\ref{th-bal1}(i$_3$) with
$\sigma:=\zeta^A\in\mathcal E^+$ shows that, indeed, $\zeta^A\in\mathcal E^+(A)$ as well as that $\zeta^A$ is the orthogonal projection of itself onto $\mathcal E^+(A)$; or equivalently that $\zeta^A$ is the (unique) solution to the problem of minimizing the functional $\|\mu\|^2-2\int U^{\zeta^A}\,d\mu$, $\mu$ ranging over $\mathcal E^+(A)$. This proves the former part of the theorem, for, in consequence of (\ref{ineq1}), $\int U^{\zeta^A}\,d\mu=\int U^\zeta\,d\mu$ for all $\mu\in\mathcal E^+(A)$.\footnote{Here we have used the fact, to be often useful throughout the paper, that any $\mu$-measurable subset of $A$ of inner capacity zero is $\mu$-negligible for any $\mu\in\mathcal E^+(A)$, cf.\ \cite[Lemma~2.5]{Z-arx-22}.}

For the latter part, assume that the equality $U^{\zeta^A}=U^\zeta$ n.e.\ on $A$ is also fulfilled for some $\mu_0\in\mathcal E^+(A)$ in place of $\zeta^A$. Then, by the strengthened
version of countable subadditivity for inner capacity (see Lemma~\ref{str-sub} below),
\[U^{\mu_0}=U^\zeta=U^{\zeta^A}\quad\text{n.e.\ on $A$},\]
whence $\mu_0=(\zeta^A)^A$, by virtue of Theorem~\ref{th-bal1}(ii$_3$) with $\sigma:=\zeta^A\in\mathcal E^+$. Combining this with   $(\zeta^A)^A=\zeta^A$ we obtain $\mu_0=\zeta^A$, thereby completing
the whole proof.\end{proof}

\begin{lemma}\label{str-sub}
 For arbitrary $Q\subset\mathbb R^n$ and universally measurable $U_j\subset\mathbb R^n$,
\[c_*\Bigl(\bigcup_{j\in\mathbb N}\,Q\cap U_j\Bigr)\leqslant\sum_{j\in\mathbb N}\,c_*(Q\cap U_j).\]
\end{lemma}

\begin{proof}
See \cite[pp.~157--158]{F1} (for $\alpha=2$, cf.\ \cite[p.~253]{Ca2}); compare with \cite[p.~144]{L}.
\end{proof}

\begin{open}
It is still unknown whether $\zeta^A$, $\zeta\in\mathfrak M^+$ being arbitrary, is concentrated on $A$ (unless, of course, the set $A$ is closed, or $\zeta^A\in\mathcal E^+$, cf.\ Theorem~\ref{l-oo'}). What we can prove so far, is the following.

\begin{theorem}\label{l-oo''}
$\zeta^A\in\mathfrak M^+(A)$ does hold if $A$ is quasiclosed, while $\zeta$ is bounded.
\end{theorem}

\begin{proof} Define $\mathfrak M^+_q:=\{\mu\in\mathfrak M^+:\ \mu(\mathbb R^n)\leqslant q\}$, where $q:=\zeta(\mathbb R^n)<\infty$. Then $\mathfrak M^+_q$ is vaguely compact \cite[Section~III.1.9, Corollary~2]{B2} and hereditary \cite[Definition~5.2]{Fu4}.
On the other hand, $\zeta^A$ is the vague limit of the net
$(\zeta^K)_{K\in\mathfrak C_A}\subset\mathfrak M^+(A)$, see (\ref{conv1}). Since
$\zeta^K(\mathbb R^n)\leqslant\zeta(\mathbb R^n)=q$, see (b), we actually have $(\zeta^K)_{K\in\mathfrak C_A}\subset\mathfrak M^+_q\cap\mathfrak M^+(A)$.
Applying \cite[Corollary~6.2]{Fu4} with $\mathcal J:=\mathfrak M^+_q$ and $H:=A$, we therefore conclude that $(\zeta^K)_{K\in\mathfrak C_A}$ has a vague limit point $\nu_0\in\mathfrak M^+_q\cap\mathfrak M^+(A)$. The vague topology on $\mathfrak M^+$ being Hausdorff, $\nu_0=\zeta^A$, whence the claim.
\end{proof}

\begin{corollary}
  If $A$ is quasiclosed, then for any $x\in\mathbb R^n$, the inner $\alpha$-harmonic measure $\varepsilon_x^A$ is concentrated on $A$.\footnote{See \cite[Theorem~4.1]{Z-bal2} for a description of the support $S(\varepsilon_x^A)$. For refined properties of $\varepsilon_x^A$ in the case where $\alpha<2$ while $A$ is the complement of an open, bounded set with Lipschitz boundary, see Bogdan \cite[Lemma~6]{KB}; compare with Section~\ref{sec-harm} below.}
  \end{corollary}
\end{open}

\subsection{Inner $\alpha$-regular and $\alpha$-irregular points}\label{sec-reg} A point $y\in\overline{A}$ is said to be {\it inner $\alpha$-irregular} for $A$ if $\varepsilon_y^A\ne\varepsilon_y$; we denote by $A^i$ the set of all those $y$. By the Wiener type criterion \cite[Theorem~6.4]{Z-bal}, $A^i$ consists of all $y\in\overline{A}$ such that
\begin{equation}\label{w}\sum_{j\in\mathbb N}\,\frac{c_*(A_j)}{q^{j(n-\alpha)}}<\infty,\end{equation}
where $q\in(0,1)$ and $A_j:=A\cap\{x\in\mathbb R^n:\ q^{j+1}<|x-y|\leqslant q^j\}$ (and hence $A^i\subset\partial A$); while by the Kel\-logg--Ev\-ans type theorem \cite[Theorem~6.6]{Z-bal},\footnote{Observe that both (\ref{w}) and (\ref{KE}) refer to inner capacity; compare with the Kel\-l\-ogg--Ev\-ans and Wiener type theorems established for {\it outer} $\alpha$-irregular points (see e.g.\ \cite{BH,Br,Ca2,Doob}). Regarding (\ref{KE}), we also remark that the whole set $A^i$ may be of nonzero capacity \cite[Section~V.4.12]{L}.}
\begin{equation}\label{KE}c_*(A\cap A^i)=0.\end{equation}

All other points of $\overline{A}$ are said to be {\it inner $\alpha$-regular} for $A$; we denote $A^r:=\overline{A}\setminus A^i$. Similarly as in \cite[p.~286]{L}, applying (\ref{alt}) with $\zeta:=\varepsilon_y$ gives
\[y\in A^r\iff U^{\mu^A}(y)=U^\mu(y)\quad\text{for all $\mu\in\mathfrak M^+$}.\]

\subsection{Basic facts about inner equilibrium measure}\label{sec-eq} In the current research, the inner ($\alpha$-Riesz) equilibrium measure $\gamma_A$ of $A\subset\mathbb R^n$ is understood in an extended sense where $I(\gamma_A)$ along with $\gamma_A(\mathbb R^n)$ might be $+\infty$.
Define
\begin{equation*}
\Gamma_A:=\bigl\{\mu\in\mathfrak M^+:\ U^\mu\geqslant1\quad\text{n.e.\ on $A$}\bigr\}.\end{equation*}

\begin{definition}[{\rm see \cite[Section~5]{Z-bal}}]\label{def-eq} $\gamma_A\in\mathfrak M^+$ is said to be {\it the inner equilibrium measure} of $A$ if it is of minimum potential in the class $\Gamma_A$, that is, if $\gamma_A\in\Gamma_A$ and
\begin{equation}\label{G'}U^{\gamma_A}=\min_{\mu\in\Gamma_A}\,U^\mu\quad\text{on $\mathbb R^n$}.\end{equation}
\end{definition}

Applying \cite[Theorem~1.12]{L} shows that such $\gamma_A$ is unique (if it exists).

As seen from the following theorem, the concept of inner equilibrium measure is closely related to that of inner $\alpha$-harmonic measure. For any $y\in\mathbb R^n$, we denote by $A_y^*$ the inverse of $A$ with respect to the sphere $S_{y,1}:=\partial B_{y,1}$, cf.\ \cite[Section~IV.5.19]{L}.

\begin{theorem}\label{th-eq} For arbitrary $A\subset\mathbb R^n$, the following {\rm(i$_5$)}--{\rm(vi$_5$)} are equivalent.
\begin{itemize}
  \item[{\rm(i$_5$)}] There exists the {\rm(}unique{\rm)} inner equilibrium measure $\gamma_A$ of $A$.
  \item[{\rm(ii$_5$)}] There exists $\nu\in\mathfrak M^+$ having the property
\[\essinf_{x\in A}\,U^\nu(x)>0,\]
the infimum being taken over all of $A$ except for a subset of $c_*(\cdot)=0$.
\item[{\rm(iii$_5$)}]For some {\rm(}equivalently, every{\rm)} $y\in\mathbb R^n$,
\begin{equation}\label{iii}\sum_{j\in\mathbb N}\,\frac{c_*(A_j)}{q^{j(n-\alpha)}}<\infty,\end{equation}
where $q\in(1,\infty)$ and $A_j:=A\cap\{x\in\mathbb R^n:\ q^j\leqslant|x-y|<q^{j+1}\}$.
\item[{\rm(iv$_5$)}]For some {\rm(}equivalently, every{\rm)} $y\in\mathbb R^n$,
\begin{equation*}y\in(A_y^*)^{rc}\quad\bigl({}:=\mathbb R^n\setminus(A_y^*)^r\bigr).\end{equation*}
\item[{\rm(v$_5$)}]For some {\rm(}equivalently, every{\rm)} $y\in\mathbb R^n$, $\varepsilon_y^{A_y^*}$ is $C$-absolutely continuous.\footnote{$\mu\in\mathfrak M^+$ is said to be {\it $C$-absolutely continuous} if $\mu(K)=0$ for every compact set $K\subset\mathbb R^n$ with $c(K)=0$. This certainly occurs if $I(\mu)<\infty$, but not conversely (see \cite[pp.~134--135]{L}).\label{f-C}}
\item[{\rm(vi$_5$)}] There exists $\chi\in\mathfrak M^+$ having the property $\chi^A(\mathbb R^n)<\chi(\mathbb R^n)$.\footnote{Compare with Section~\ref{some}, (b).}
\end{itemize}

Furthermore, if any one of these {\rm(i$_5$)--(vi$_5$)} is fulfilled, then for every $y\in\mathbb R^n$,
\begin{equation}\label{har-eq}\varepsilon_y^{A_y^*}=(\gamma_A)^*,\end{equation}
where $(\gamma_A)^*$ is the Kelvin transform of $\gamma_{A}\in\mathfrak M^+$ with respect to the sphere $S_{y,1}$.\footnote{For the concept of Kelvin transformation, see \cite[Section~14]{R} as well as \cite[Section~IV.5.19]{L}.}
\end{theorem}

\begin{proof}
See \cite{Z-bal2} (Theorem~2.1 and Corollary~5.3).
\end{proof}

\begin{definition}[{\rm cf.\ \cite[Definition~2.1]{Z-bal2}}]\label{def-thin} $A\subset\mathbb R^n$ is said to be {\it inner $\alpha$-thin at infinity} if any one of the above {\rm(i$_5$)--(vi$_5$)} is fulfilled.
\end{definition}

\begin{remark}\label{doob}
As seen from (\ref{iii}), the concept of inner $\alpha$-thinness of a set at infinity thus defined coincides with that introduced by Kurokawa and Mizuta \cite{KM}. If $\alpha=2$ while a set in question is Borel, then this concept also coincides with that of {\it outer} thinness of a set at infinity, introduced by Doob \cite[pp.~175--176]{Doob}.
\end{remark}

Assume that the inner equilibrium measure $\gamma_A$ exists (or equivalently that $A$ is inner $\alpha$-thin at infinity). By use of properties of the Kelvin transformation (see \cite[Section~IV.5.19]{L}), we derive from Theorem~\ref{th-eq} and the facts reviewed in Sections~\ref{some}, \ref{sec-reg} that $\gamma_A$ is $C$-absolutely continuous, supported by $\overline{A}$, and having the properties\footnote{However, the equality $U^{\gamma_A}=1$ n.e.\ on $A$ does not determine $\gamma_A$ uniquely (unless, of course, the set $A$ is closed, see \cite[p.~178, Remark]{L}); compare with Theorem~\ref{th-eq''}.}
\begin{align}
\label{eqne} &U^{\gamma_A}=1\quad\text{on $A^r$ (hence, n.e.\ on $A$)},\\
\label{ineqne} &0<U^{\gamma_A}\leqslant1\quad\text{on $\mathbb R^n$}.
\end{align}
The same $\gamma_A$ can be uniquely determined by either of the two limit relations
\begin{align}
  \gamma_K&\to\gamma_A\quad\text{vaguely in $\mathfrak M^+$ as $K\uparrow A$},\label{cv1}\\
  U^{\gamma_K}&\uparrow U^{\gamma_A}\quad\text{pointwise on $\mathbb R^n$ as $K\uparrow A$},\label{cv2}
\end{align}
where $\gamma_K$ denotes the only measure in $\mathcal E^+(K)$ with $U^{\gamma_K}=1$ n.e.\ on $K$~--- namely, the (classical) equilibrium measure on $K$ \cite[Section~II.1.3]{L}, normalized by
\begin{equation}\label{k}
\gamma_K(\mathbb R^n)=\|\gamma_K\|^2=c(K).
\end{equation}
(For the alternative characterizations (\ref{cv1}) and (\ref{cv2}) of $\gamma_A$, see \cite[Lemma~5.3]{Z-bal}.)

If moreover $c_*(A)<\infty$, then (and only then) the above $\gamma_A$ is of finite energy,\footnote{In fact, the "if" part of this claim is obtained from (\ref{cv1}) and (\ref{k}) by use of \cite[Eq.~(1.4.5)]{L} (the principle of descent), whereas the "only if" part follows from the chain of inequalities
\[c(K)=\int U^{\gamma_A}\,d\gamma_K=\int U^{\gamma_K}\,d\gamma_A\leqslant\int U^{\gamma_A}\,d\gamma_A=I(\gamma_A)<\infty\quad\text{for all $K\in\mathfrak C_A$},\]
where the first equality is implied by (\ref{eqne}) and (\ref{k}), while the first inequality~--- by (\ref{cv2}).}
and it can alternatively be found as the only measure in $\Gamma_A\cap\mathcal E$ of minimum energy:
\begin{equation}\label{cv4'}\|\gamma_A\|^2=\min_{\mu\in\Gamma_A\cap\mathcal E}\,\|\mu\|^2\end{equation}
(see \cite[Theorems~6.1, 9.2]{Z-arx-22}). Furthermore, then, along with (\ref{cv1}) and (\ref{cv2}),
\begin{equation*}
\gamma_K\to\gamma_A\quad\text{strongly in $\mathcal E^+$ as $K\uparrow A$}
\end{equation*}
(see \cite[Theorem~8.1]{Z-arx-22}), and therefore
\[\gamma_A\in\mathcal E'(A).\]

On account of \cite[Theorem~7.2(c)]{Z-arx-22}, we are thus led to the following conclusion.

\begin{theorem}\label{th-eq''}Assume $c_*(A)<\infty$ and $(\mathcal P)$ holds. Then the inner equilibrium measure $\gamma_A$, uniquely determined by either of {\rm(\ref{G'})} or {\rm(\ref{cv4'})}, belongs to $\mathcal E^+(A)$:
\begin{equation}\label{eqfin}
\gamma_A\in\mathcal E^+(A),
\end{equation}
and it is characterized as the only measure in $\mathcal E^+(A)$ with $U^{\gamma_A}=1$ n.e.\ on $A$.\end{theorem}

\begin{corollary}\label{cor-eq}
 Under the assumptions of Theorem~{\rm\ref{th-eq''}},
 \begin{equation}\label{eqfin'}(\gamma_A)^A=\gamma_A.\end{equation}
 \end{corollary}

\begin{proof} Substituting (\ref{ee}) into Theorem~\ref{th-bal1}(i$_3$) with $\sigma:=\gamma_A$ shows that $(\gamma_A)^A$ is actually the orthogonal projection of $\gamma_A$ in the pre-Hil\-bert space $\mathcal E$ onto the convex, strongly complete cone $\mathcal E^+(A)$. As $\gamma_A\in\mathcal E^+(A)$ by (\ref{eqfin}), (\ref{eqfin'}) follows.
\end{proof}

\subsection{Some basic facts about inner $\alpha$-harmonic measure}\label{sec-harm}
Being a natural generalization of the classical concept of ($2$-)har\-mo\-nic measure, the inner $\alpha$-har\-mon\-ic measure $\varepsilon_y^A$ serves as the main tool in solving the generalized Dirichlet problem for $\alpha$-har\-mon\-ic functions, cf.\ \cite{BH,KB',L}. Besides, due to the formula (see \cite[Theorem~5.1]{Z-bal2})
\[\zeta^A=\int\varepsilon_y^A\,d\zeta(y),\]
it is useful in the study of the inner balayage $\zeta^A$ for arbitrary $\zeta\in\mathfrak M^+$, see~\cite{Z-bal2}.

The following theorem on the total mass of the inner $\alpha$-harmonic measure $\varepsilon_y^A$, $y\in\mathbb R^n$ and $A\subset\mathbb R^n$ being arbitrary, has already been utilized in Section~\ref{sec-main} above.

\begin{theorem}[{\rm see \cite[Theorem~2.2]{Z-bal2}}]\label{harm-tot}
 For any $y\in\mathbb R^n$ and any $A\subset\mathbb R^n$,\footnote{Regarding the latter relation in (\ref{TM}), see Theorem~\ref{th-eq}(vi$_5$).}
 \begin{equation}\label{TM}\varepsilon_y^A(\mathbb R^n)=\left\{
\begin{array}{cl}U^{\gamma_A}(y)&\text{if $A$ is inner $\alpha$-thin at infinity},\\
1&\text{otherwise},\\ \end{array} \right.
\end{equation}
$\gamma_A$ being the inner equilibrium measure of $A$ {\rm(see Section~\ref{sec-eq})}.
\end{theorem}

Fix $y\in\mathbb R^n$. As seen from Theorem~\ref{th-eq} with $A$ and $A_y^*$ interchanged, $\varepsilon_y^A$ is $C$-ab\-s\-ol\-ut\-ely continuous if and only if $A_y^*$ is inner $\alpha$-thin at infinity, or equivalently $y\notin A^r$, cf.\ \cite[Corollary~2.1]{Z-bal2}. We shall, however, mainly be concerned with $y\notin A^r$ such that
the inner $\alpha$-harmonic measure $\varepsilon_y^A$ is of finite energy:
\begin{equation*}
\varepsilon_y^A\in\mathcal E^+.
\end{equation*}
Noting from Section~\ref{some}, see (a) and (b), that this necessarily holds true if $y\notin\overline{A}$ (see footnote~\ref{f-proof} for details), we introduce the following (new) concept.

\begin{definition}\label{def-ultra}
Given $A\subset\mathbb R^n$, a point $y\in\overline{A}$ is said to be inner {\it $\alpha$-ultrairregular} for $A$ if $\varepsilon_y^A\in\mathcal E^+$. Let $A^u$ stand for the set of all those $y$; then certainly
\[A^u\subset A^i\subset\partial A.\]
\end{definition}

\begin{theorem}\label{th-eq'} Given $A\subset\mathbb R^n$ and $y\in\overline{A}$, the following {\rm(i$_6$)}--{\rm(iv$_6$)} are equivalent.
\begin{itemize}
  \item[{\rm(i$_6$)}] $y$ is inner $\alpha$-ultrairregular for $A$.
  \item[{\rm(ii$_6$)}] If $A_j:=A\cap\{x\in\mathbb R^n:\ q^{j+1}\leqslant|x-y|<q^j\}$, where $q\in(0,1)$, then
      \begin{equation}\label{iii'}\sum_{j\in\mathbb N}\,\frac{c_*(A_j)}{q^{2j(n-\alpha)}}<\infty.\end{equation}
\item[{\rm(iii$_6$)}] $A^*_y$, the inverse of $A$ with respect to the sphere $S_{y,1}$, is of finite inner capacity:
\[c_*(A^*_y)<\infty.\]
\item[{\rm(iv$_6$)}] $A^*_y$ is inner $\alpha$-ultrathin at infinity {\rm (cf.\ \cite[Definition~2.2]{Z-bal2})}.
\end{itemize}

Furthermore, if any one of these {\rm(i$_6$)}--{\rm(iv$_6$)} is fulfilled, then the inner $\alpha$-harmonic measure $\varepsilon_y^A$ is representable in the form
\[\varepsilon_y^A=(\gamma_{A^*_y})^*,\]
$(\gamma_{A^*_y})^*$ being the Kelvin transform of the inner equilibrium measure $\gamma_{A^*_y}\in\mathcal E^+$ with respect to the sphere $S_{y,1}$.
\end{theorem}

\begin{proof}
On account of Definition~\ref{def-ultra}, this follows from (\ref{har-eq}) and \cite[Theorem~2.3]{Z-bal2} with $A$ and $A_y^*$ interchanged.
\end{proof}

\begin{remark}
Comparing (\ref{iii'}) with (\ref{w}) we deduce that, in general,
\[A^u\Subset A^i.\] It is also worth noting that, if $\alpha=2$ while $A$ is Borel, then $y\in A^u$ if and only if $A^*_y$ is outer $2$-thin at infinity in the sense of Brelot \cite[p.~313]{Brelot}, cf.\ \cite[Remark~2.3]{Z-bal2}.
\end{remark}

\begin{example}\label{ex} Let $n=3$ and $\alpha=2$. Consider the rotation body
\begin{equation*}\Delta:=\bigl\{x\in\mathbb R^3: \ 0\leqslant x_1\leqslant1, \
x_2^2+x_3^2\leqslant\exp(-2x_1^{-\beta})\bigr\},\quad\beta\in(0,\infty).\end{equation*}
For any $\beta>0$, the origin $x=0$ is $2$-irregular for $\Delta$ (Lebesgue cusp), and moreover
\[x=0\text{\ is $2$-ultrairregular for $\Delta$}\iff\beta>1.\]
(Compare with \cite[Section~V.1.3, Example]{L} and \cite[Example~2.1]{Z-bal2}.)
\end{example}

\section{Proofs of the main results: preparatory lemmas}\label{sec-prep}

In the rest of this paper, we fix a set $A\subset\mathbb R^n$, $\overline{A}\ne\mathbb R^n$, satisfying (\ref{nonzero}) and $(\mathcal P)$, and $z\in A^u\cup{\overline A}^c$. Then the inner $\alpha$-harmonic measure $\varepsilon_z^A$ is clearly of finite energy:\footnote{Indeed, for $z\in A^u$, this is obvious from Definition~\ref{def-ultra}, while for $z\in{\overline A}^c$, the Euclidean distance $\text{\rm dist}(z,{\overline A})$ between $\{z\}$ and ${\overline A}$, is ${}>0$, whence
\[\int U^{\varepsilon_z^A}\,d\varepsilon_z^A\leqslant\int U^{\varepsilon_z}\,d\varepsilon_z^A\leqslant\varepsilon_z^A(\mathbb R^n)/{\rm dist}(z,{\overline A})^{n-\alpha}\leqslant1/{\rm dist}(z,{\overline A})^{n-\alpha}<\infty,\]
the first and the third inequalities being seen from Section~\ref{some}, (a) and (b), respectively.\label{f-proof}}
\begin{equation}\label{Fen}
\varepsilon_z^A\in\mathcal E^+,
\end{equation}
hence
\begin{equation}\label{integr}\int U^{\varepsilon_z}\,d\mu=\int U^{\varepsilon_z^A}\,d\mu=\langle\varepsilon_z^A,\mu\rangle<\infty\quad\text{for all $\mu\in\mathcal E^+(A)$},\end{equation}
the former equality being obtained from (\ref{ineq1}) with $\zeta:=\varepsilon_z$ by use of the fact that any $\mu$-measurable subset of $A$ of inner capacity zero is $\mu$-negligible for any $\mu\in\mathcal E^+(A)$.

Fix $q\in(0,\infty)$. In view of $(\mathcal P)$ and (\ref{Fen}), Theorem~\ref{l-oo'} with $\zeta:=q\varepsilon_z$ gives
\begin{equation}\label{FenA}
q\varepsilon_z^A\in\mathcal E^+(A).
\end{equation}
Moreover, by virtue of the same Theorem~\ref{l-oo'}, $q\varepsilon_z^A$ solves the problem of minimizing the Gauss functional $I_{f_{q,z}}(\mu)$ over $\mu\in\mathcal E^+(A)$:
\begin{equation}\label{qbal}
-q^2\|\varepsilon_z^A\|^2=I_{f_{q,z}}(q\varepsilon_z^A)=\min_{\mu\in\mathcal E^+(A)}\,I_{f_{q,z}}(\mu)=:\widehat{w}_{f_{q,z}}(A),
\end{equation}
the first equality being derived from (\ref{G}), (\ref{integr}), and (\ref{FenA}) by a direct verification.

Noting that, obviously,
\begin{equation}\label{WWin}
\widehat{w}_{f_{q,z}}(A)\leqslant w_{f_{q,z}}(A),
\end{equation}
we infer from (\ref{qbal}) that $w_{f_{q,z}}(A)>-\infty$, and hence Problem~\ref{pr} makes sense.

Lemma~\ref{l-char}, providing characteristic properties of the solution to Problem~\ref{pr}, will be useful in subsequent proofs.

\begin{lemma}\label{l-char}
For $\lambda\in\breve{\mathcal E}^+(A)$ to serve as the {\rm(}unique{\rm)} solution to Problem~{\rm\ref{pr}}, it is necessary and sufficient that either of the following two relations be fulfilled:
\begin{align}
  U^\lambda_{f_{q,z}}&\geqslant C_\lambda\quad\text{n.e.\ on $A$},\label{ch1}\\
 U^\lambda_{f_{q,z}}&=C_\lambda\quad\text{$\lambda$-a.e.\ on $\mathbb R^n$},\label{ch2}
\end{align}
where
\begin{equation}\label{const}
C_\lambda:=\int U^\lambda_{f_{q,z}}\,d\lambda\in(-\infty,\infty).
\end{equation}
\end{lemma}

\begin{proof}
  See \cite{Z5a} (Theorem~1 and Proposition~1).
\end{proof}

\subsection{Extremal measures}We call a net $(\mu_s)_{s\in S}\subset\breve{\mathcal E}^+(A)$ {\it minimizing} if
\begin{equation}\label{min}
\lim_{s\in S}\,I_{f_{q,z}}(\mu_s)=w_{f_{q,z}}(A);
\end{equation}
let $\mathbb M_{f_{q,z}}(A)$ stand for the (nonempty) set of all those $(\mu_s)_{s\in S}$.

\begin{lemma}\label{l-ext00}There exists the unique $\xi_{A,f_{q,z}}\in\mathcal E^+$ such that, for every minimizing net $(\mu_s)_{s\in S}\in\mathbb M_{f_{q,z}}(A)$,
\begin{equation}\label{ext}
 \mu_s\to\xi_{A,f_{q,z}}\quad\text{strongly and vaguely in $\mathcal E^+$}.
\end{equation}
This $\xi_{A,f_{q,z}}$ is referred to as the extremal measure {\rm(}in Problem~{\rm\ref{pr}}{\rm)}.
\end{lemma}

\begin{proof}
This follows in a manner similar to that in \cite[Lemma~4.2]{Z-Rarx}, by use of the convexity of the class $\breve{\mathcal E}^+(A)$, the parallelogram identity in the pre-Hil\-bert space $\mathcal E$, the strict positive definiteness of the Riesz kernel, and the finiteness of $w_{f_{q,z}}(A)$.
\end{proof}

As seen from (\ref{ext}), we actually have
\begin{equation}\label{ext''}
\xi_{A,f_{q,z}}\in\mathcal E^+(A),
\end{equation}
the class $\mathcal E^+(A)$ being strongly closed according to $(\mathcal P)$. Furthermore, again by (\ref{ext}),
\begin{equation}\label{leq}
 \xi_{A,f_{q,z}}(\mathbb R^n)\leqslant1,
\end{equation}
the mapping $\mu\mapsto\mu(\mathbb R^n)$ being vaguely l.s.c.\ on $\mathfrak M^+$ \cite[Section~IV.1, Proposition~4]{B2}.

\begin{lemma}\label{l-ext1} $\xi_{A,f_{q,z}}$ serves as the solution $\lambda_{A,f_{q,z}}$ to Problem~{\rm\ref{pr}} if and only if
\begin{equation}\label{e-ext1}
 \xi_{A,f_{q,z}}(\mathbb R^n)=1.
\end{equation}
\end{lemma}

\begin{proof} $I_{f_{q,z}}(\cdot)$ is strongly continuous on
the (strongly closed) set $\mathcal E^+(A)$, because
\begin{equation*}
I_{f_{q,z}}(\mu)=\|\mu\|^2-2\langle q\varepsilon_z^A,\mu\rangle=\|\mu-q\varepsilon_z^A\|^2-q^2\|\varepsilon_z^A\|^2,\quad\mu\in\mathcal E^+(A),
\end{equation*}
which is obtained by substituting (\ref{integr}) into (\ref{G}). Therefore, in view of (\ref{ext}),
\[I_{f_{q,z}}\bigl(\xi_{A,f_{q,z}}\bigr)=\lim_{s\in S}\,I_{f_{q,z}}(\mu_s),\]
where $(\mu_s)_{s\in S}\in\mathbb M_{f_{q,z}}(A)$, and combining this with (\ref{min}) gives
\begin{equation}\label{e-ext}
 I_{f_{q,z}}\bigl(\xi_{A,f_{q,z}}\bigr)=w_{f_{q,z}}(A).
\end{equation}

Now, the "if" part of the lemma follows at once from (\ref{ext''}), (\ref{e-ext1}), and (\ref{e-ext}), whereas the "only if" part is obvious, for the trivial sequence $(\lambda_{A,f_{q,z}})$ is certainly minimizing, and hence converges strongly and vaguely to both $\xi_{A,f_{q,z}}$ and $\lambda_{A,f_{q,z}}$.\end{proof}

Henceforth, when speaking of a (compact) set $K\in\mathfrak C_A$, we always understand it to follow some $K_0\in\mathfrak C_A$ with $c(K_0)>0$, which is possible due to assumption (\ref{nonzero}).

\begin{lemma}\label{l-ext2}Given $z\in A^u\cup\overline{A}^c$, Problem~{\rm\ref{pr}} is solvable for all $K\in\mathfrak C_A$, and moreover those solutions $\lambda_{K,f_{q,z}}$ form a minimizing net:
\begin{equation}\label{min2}
 \bigl(\lambda_{K,f_{q,z}}\bigr)_{K\in\mathfrak C_A}\in\mathbb M_{f_{q,z}}(A).
\end{equation}
\end{lemma}

\begin{proof} Since for given $z\in A^u\cup\overline{A}^c$,
\begin{equation}\label{Fen'}
\varepsilon_z^K=\bigl(\varepsilon_z^A\bigr)^K\in\mathcal E^+\quad\text{for all $K\in\mathfrak C_A$},
\end{equation}
Lemma~\ref{l-ext1} remains valid for each $K\in\mathfrak C_A$ in place of $A$. (The former relation in (\ref{Fen'}) is clear from Section~\ref{some}, (c), and the latter from (\ref{Fen}) and Theorem~\ref{th-bal1}.)

By Lemma~\ref{l-ext1} applied to $K\in\mathfrak C_A$, $\lambda_{K,f_{q,z}}$ exists if (and only if)
$\xi_{K,f_{q,z}}(\mathbb R^n)=1$, which indeed holds true, for $\xi_{K,f_{q,z}}$ is the vague limit of any net $(\mu_s)_{s\in S}\in\mathbb M_{f_{q,z}}(K)$, whereas $\{\mu\in\mathfrak M^+(K): \mu(K)=1\}$ is vaguely closed \cite[Section~III.1.9, Corollary~3]{B2}.

Since the net $\bigl(w_{f_{q,z}}(K)\bigr)_{K\in\mathfrak C_A}$ obviously decreases and has $w_{f_{q,z}}(A)$ as a lower bound, the proof of (\ref{min2}) is reduced to that of the inequality
 \begin{equation}\label{Le}
 \lim_{K\uparrow A}\,w_{f_{q,z}}(K)\leqslant w_{f_{q,z}}(A).
 \end{equation}
 For any $\mu\in\breve{\mathcal E}^+(A)$, $\mu|_K(\mathbb R^n)\uparrow1$ as $K\uparrow A$, whence
 \begin{equation}\label{Lee}I_{f_{q,z}}(\mu)=\lim_{K\uparrow A}\,I_{f_{q,z}}(\mu|_K)=\lim_{K\uparrow A}\,I_{f_{q,z}}(\nu_K)\geqslant\lim_{K\uparrow A}\,w_{f_{q,z}}(K),\end{equation}
 where
 \[\nu_K:=\mu|_K/\mu|_K(\mathbb R^n)\in\breve{\mathcal E}^+(K),\] and (\ref{Le}) follows at once from (\ref{Lee}) by letting $\mu$ range over $\breve{\mathcal E}^+(A)$. (The former equality in (\ref{Lee}) holds true by \cite[Lemma~1.2.2]{F1}, applied to each of the $\mu$-in\-t\-eg\-r\-able functions $\kappa_\alpha$ and $f_{q,z}$, the $\mu$-integrability of $f_{q,z}$ being seen from (\ref{integr}).)
\end{proof}

\begin{lemma}\label{l-ext3}For the extremal measure $\xi:=\xi_{A,f_{q,z}}$, we have
\begin{equation}\label{Xi}
U^\xi_{f_{q,z}}\geqslant C_\xi\quad\text{n.e.\ on $A$},
\end{equation}
where
\begin{equation}\label{Cxi}
C_\xi:=\int U^\xi_{f_{q,z}}\,d\xi\in(-\infty,\infty).
\end{equation}
\end{lemma}

\begin{proof}The finiteness of the constant $C_\xi$, introduced in (\ref{Cxi}), is obvious from (\ref{ext''}) and the $\xi$-integrability of $f_{q,z}$, the latter being clear from (\ref{integr}) applied to $\xi$.

As for (\ref{Xi}), it is enough to verify this inequality n.e.\ on $K_*$, where $K_*\in\mathfrak C_A$ is arbitrarily given. Choose a subsequence $\bigl(\lambda_{K_j,f_{q,z}}\bigr)$ of the net $\bigl(\lambda_{K,f_{q,z}}\bigr)_{K\in\mathfrak C_A}$ such that
\begin{equation}\label{xiconv}
 \lambda_{K_j,f_{q,z}}\to\xi\quad\text{strongly in $\mathcal E^+$ as $j\to\infty$},
\end{equation}
which is possible due to (\ref{ext}), (\ref{min2}), and the first-countability of the strong topology on $\mathcal E$. We can assume that $K_*\subset K_j$ for all $j$; for if not, we replace $\bigl(\lambda_{K_j,f_{q,z}}\bigr)$ by the sequence $\bigl(\lambda_{K_j\cup K_*,f_{q,z}}\bigr)$, which remains minimizing in view of the monotonicity of the net $\bigl(w_{f_{q,z}}(K)\bigr)_{K\in\mathfrak C_A}$, and hence satisfies (\ref{xiconv}) as well. Furthermore, there is no loss of generality in assuming that
\begin{equation}\label{pconv}
U^{\lambda_{K_j,f_{q,z}}}\to U^\xi\quad\text{pointwise n.e.\ on $K_*$ as $j\to\infty$},
\end{equation}
which is seen from (\ref{xiconv}) by making use of \cite[p.~166, Remark]{F1}.

But, by Lemma~\ref{l-char} applied to each of those $K_j$,
\begin{equation}\label{Kj}
 U_{f_{q,z}}^{\lambda_{K_j,f_{q,z}}}\geqslant\int U_{f_{q,z}}^{\lambda_{K_j,f_{q,z}}}\,d\lambda_{K_j,f_{q,z}}\quad\text{n.e.\ on $K_j$},
\end{equation}
and therefore n.e.\ on $K_*$. In consequence of (\ref{xiconv}),
\begin{align}
\notag\lim_{j\to\infty}\,\Bigl|\int f_{q,z}\,d\bigl(\lambda_{K_j,f_{q,z}}-\xi\bigr)\Bigr|&=q\lim_{j\to\infty}\,\bigl|\bigl\langle\varepsilon_z^A, \lambda_{K_j,f_{q,z}}-\xi\bigr\rangle\bigr|\\{}&\leqslant q\|\varepsilon_z^A\|\lim_{j\to\infty}\,\|\lambda_{K_j,f_{q,z}}-\xi\|=0
\label{mec}\end{align}
as well as
\begin{equation*}
\lim_{j\to\infty}\,\|\lambda_{K_j,f_{q,z}}\|^2=\|\xi\|^2,
\end{equation*}
which taken together gives
\begin{equation}\label{Lim}
\lim_{j\to\infty}\,\int U_{f_{q,z}}^{\lambda_{K_j,f_{q,z}}}\,d\lambda_{K_j,f_{q,z}}=\int U^\xi_{f_{q,z}}\,d\xi.
\end{equation}
(In (\ref{mec}), we have used (\ref{integr}) applied to each of the measures $\lambda_{K_j,f_{q,z}}$ and $\xi$, elements of the class $\mathcal E^+(A)$, cf.\ (\ref{ext''}), as well as the Cauchy--Schwarz (Bunyakovski) inequality applied to $\varepsilon_z^A$ and $\lambda_{K_j,f_{q,z}}-\xi$, elements of the pre-Hilbert space $\mathcal E$.)

Now, passing to the limits in (\ref{Kj}), on account of (\ref{Cxi}), (\ref{pconv}), (\ref{Lim}), and the countable subadditivity of inner capacity on Borel sets \cite[p.~144]{L}, we obtain
\[U^\xi_{f_{q,z}}\geqslant C_\xi\quad\text{n.e.\ on $K_*$},\]
whence (\ref{Xi}).\end{proof}

\section{Proofs of the main results}\label{sec-proofs}

\subsection{Proof of Theorem~\ref{th1}}\label{pr-th1} Assuming first that (\ref{capfin}) holds, we need to verify the solvability of Problem~\ref{pr} for any $q\in(0,\infty)$. To this end, it is enough to show that the extremal measure $\xi_{A,f_{q,z}}$ satisfies (\ref{e-ext1}), see Lemma~\ref{l-ext1}. But under the assumptions $(\mathcal P)$ and (\ref{capfin}), the class $\breve{\mathcal E}^+(A)$ is strongly complete \cite[Theorem~3.6]{Z-Rarx}, whereas $\xi_{A,f_{q,z}}$ is the strong limit of a minimizing net $(\mu_s)_{s\in S}\in\mathbb M_{f_{q,z}}(A)$, see (\ref{ext}). Noting that $(\mu_s)_{s\in S}\subset\breve{\mathcal E}^+(A)$, we therefore get (\ref{e-ext1}), whence $\xi_{A,f_{q,z}}=\lambda_{A,f_{q,z}}$.

Assume now that $q:=H_z$, $H_z$ being introduced by (\ref{H}), or equivalently by (\ref{n4}). Since $H_z\varepsilon_z^A(\mathbb R^n)=1$, we have, on account of (\ref{FenA}),
\begin{equation}\label{hz}
 H_z\varepsilon_z^A\in\breve{\mathcal E}^+(A),
\end{equation}
whence
\begin{equation}\label{hzz}w_{f_{H_z,z}}(A)\leqslant I_{f_{H_z,z}}\bigl(H_z\varepsilon_z^A\bigr)=\widehat{w}_{f_{H_z,z}}(A)\leqslant w_{f_{H_z,z}}(A),\end{equation}
the equality and the subsequent inequality being obtained from (\ref{qbal}) and (\ref{WWin}), respectively, with $q=H_z$. As seen from (\ref{hz}) and (\ref{hzz}), Problem~\ref{pr} for $q=H_z$ is indeed solvable, and moreover the solution $\lambda_{A,f_{H_z,z}}$ is given by the formula\footnote{Thus $w_{f_{H_z,z}}(A)=H_z^2\|\varepsilon_z^A\|^2$, see (\ref{hzz}) and (\ref{qbal}).}
\begin{equation}\label{Sol}
 \lambda_{A,f_{H_z,z}}=H_z\varepsilon_z^A.
\end{equation}
We also note that for any $A$, the inner $f_{H_z,z}$-weighted equilibrium constant equals $0$, because, in consequence of (\ref{iec}) and (\ref{Sol}),
\begin{equation}\label{iec0}c_{A,\lambda_{A,f_{H_z,z}}}=\int U^{\lambda_{A,f_{H_z,z}}}_{f_{H_z,z}}\,d\lambda_{A,f_{H_z,z}}=
H_z^2\int\bigl(U^{\varepsilon_z^A}-U^{\varepsilon_z}\bigr)\,d\varepsilon_z^A=0,\end{equation}
the last equality being derived from (\ref{ineq1}) and (\ref{FenA}) by use of the fact that any $\mu$-meas\-ur\-able subset of $A$ of inner capacity zero is $\mu$-negligible for any $\mu\in\mathcal E^+(A)$.

To complete the proof of the "if" part, it remains to consider the case where
\begin{equation}\label{SL}
q>H_z.
\end{equation}
Assume first that $C_\xi\geqslant0$, $C_\xi$ being the (finite) number introduced in (\ref{Cxi}). Then, by virtue of (\ref{Xi}), the extremal measure $\xi:=\xi_{A,f_{q,z}}$ has the property
\[U^\xi\geqslant-f_{q,z}=qU^{\varepsilon_z}\quad\text{n.e.\ on $A$},\]
whence, by use of (\ref{ineq1}) and Lemma~\ref{str-sub},
\begin{equation}\label{e-pr1}
 U^\xi\geqslant qU^{\varepsilon_z^A}\quad\text{n.e.\ on $A$}.
\end{equation}
As $\xi$ and $\varepsilon_z^A$ are of finite energy (hence, $C$-absolutely continuous) and are concentrated on $A$ (see (\ref{ext''}) and (\ref{FenA}), respectively), we conclude from (\ref{e-pr1}) by use of the principle of positivity of mass in the form stated in \cite[Theorem~1.5]{Z-Deny} that
\[\xi(\mathbb R^n)\geqslant q\varepsilon_z^A(\mathbb R^n).\]
Combining this with (\ref{leq}) and (\ref{H}) yields $q\leqslant H_z$, which however contradicts (\ref{SL}). Hence, under assumption (\ref{SL}), we necessarily have
\begin{equation}\label{Cxi0}
 C_\xi<0.
\end{equation}

Being valid n.e.\ on $A$, relation (\ref{Xi}) does hold $\xi$-a.e.\ on $\mathbb R^n$, for $\xi\in\mathcal E^+(A)$. This gives by integration
\[\int U^\xi_{f_{q,z}}\,d\xi\geqslant C_\xi\cdot\xi(\mathbb R^n),\]
whence, by (\ref{Cxi}) and (\ref{Cxi0}), $\xi(\mathbb R^n)\geqslant1$. In view of (\ref{leq}), we thus have
\[\xi(\mathbb R^n)=1,\] which according to Lemma~\ref{l-ext1} implies that $\xi$ must serve as the solution  $\lambda_{A,f_{q,z}}$.

To verify the remaining "only if" part, we only need to consider the case where
\begin{equation}\label{oif}
c_*(A)=\infty,\quad q<H_z.
\end{equation}
Because of the former assumption in (\ref{oif}), there exist $\tau_j\in\breve{\mathcal E}^+(A)$, $j\in\mathbb N$, with
\begin{equation}\label{oiff}
\|\tau_j\|<1/j,\quad S(\tau_j)\subset\{|x|\geqslant j\}.
\end{equation}
Define
\begin{equation*}
\mu_j:=q\varepsilon_z^A+M\tau_j,\quad\text{where $M:=(H_z-q)/H_z\in(0,\infty)$}.
\end{equation*}
Clearly, $\mu_j\in\breve{\mathcal E}^+(A)$ for all $j\in\mathbb N$, whence
\begin{equation}\label{ww}
I_{f_{q,z}}(\mu_j)\geqslant w_{f_{q,z}}(A).
\end{equation}
By a direct verification,
\[I_{f_{q,z}}(\mu_j)=-q^2\|\varepsilon_z^A\|^2+M^2\|\tau_j\|^2,\]
which yields, by use of the former relation in (\ref{oiff}),
\[\lim_{j\to\infty}\,I_{f_{q,z}}(\mu_j)=-q^2\|\varepsilon_z^A\|^2.\]
Combining this with (\ref{qbal}), (\ref{WWin}), and (\ref{ww}) shows that, actually,
\[\lim_{j\to\infty}\,I_{f_{q,z}}(\mu_j)=w_{f_{q,z}}(A),\]
so that
\[(\mu_j)\in\mathbb M_{f_{q,z}}(A),\]
and hence $(\mu_j)$ must converge vaguely to the extremal measure $\xi_{A,f_{q,z}}$ (Lemma~\ref{l-ext00}). But due to the latter relation in (\ref{oiff}), this vague limit is actually $q\varepsilon_z^A$, whence
\[\xi_{A,f_{q,z}}=q\varepsilon_z^A,\]
the vague topology on $\mathfrak M$ being Hausdorff. As $q<H_z$, this gives $\xi_{A,f_{q,z}}(\mathbb R^n)<1$, and hence, according to Lemma~\ref{l-ext1}, Problem~\ref{pr} is indeed unsolvable.

\subsection{Proof of Corollary~\ref{cor1}}\label{pr-cor1} Assume that $A$ is closed, $\alpha$-thin at infinity, of infinite capacity, and that $z\in A^c$ is given. Then, according to Theorem~\ref{th1}, the solution $\lambda_{A,f_{q,z}}$ to Problem~\ref{pr} exists if and only if $q\geqslant H_z$, where
\begin{equation}\label{n4'}
H_z:=1/U^{\gamma_A}(z).
\end{equation}
Here $U^{\gamma_A}$ is the equilibrium measure for $A$, see Theorem~\ref{th-eq}(i$_5$).

On the other hand, applying Theorem~\ref{th-eq}(ii$_5$)  to this $\gamma_A$ and each of the sets $A^c\cap\{|x|>r\}$, $r\in(0,\infty)$, which are obviously not $\alpha$-thin at infinity, we get
\[\limsup_{A^c\ni z\to\infty_{\mathbb R^n}}\,H_z=\infty,\]
and hence there is a sequence $(z_j)\subset A^c$ approaching $\infty_{\mathbb R^n}$ and having the property
\[\lim_{j\to\infty}\,H_{z_j}=\infty.\]
Setting $q_j:=H_{z_j}$, we therefore obtain the corollary by making use of the assertion on the solvability of Problem~\ref{pr}, reminded at the beginning of this proof.

\subsection{Proof of Corollary~\ref{cor2}}\label{pr-cor2} We begin with the following simple lemma, which will also be useful in the sequel.

\begin{lemma}\label{L-Hcon}
 For any $A\subset\mathbb R^n$ and any $y\in\overline{A}^c\cup(A^r\cap\partial A)$,
 \begin{equation}
 \lim_{x\to y}\,H_x=H_y,\label{Hcon}
\end{equation}
$H_x$, $x\in\mathbb R^n$, being introduced by means of either of {\rm(\ref{H})} or {\rm(\ref{n4})}. Hence
\begin{equation}
 \lim_{x\to y}\,H_x=1\quad\text{for any $y\in A^r\cap\partial A$}.\label{Hcon'}
\end{equation}
\end{lemma}

\begin{proof}
We can assume that $A$ is inner $\alpha$-thin at infinity, for if not, $H_x=1$ for all $x\in\mathbb R^n$, and hence (\ref{Hcon}) and (\ref{Hcon'}) hold trivially. Therefore, there exists the inner equilibrium measure $\gamma_A$ of $A$, and $H_x$, $x\in\mathbb R^n$, is now given by means of (\ref{n4'}). It is thus enough to consider the case where $y\in A^r\cap\partial A$, $U^{\gamma_A}$ being continuous on $\overline{A}^c$. But according to (\ref{eqne}), then $U^{\gamma_A}(y)=1$, which yields
\begin{equation}\label{ls}
 1=U^{\gamma_A}(y)\leqslant\liminf_{x\to y}\,U^{\gamma_A}(x),
\end{equation}
$U^{\gamma_A}$ being l.s.c.\ on $\mathbb R^n$. On the other hand, $U^{\gamma_A}\leqslant1$ on $\mathbb R^n$ (see (\ref{ineqne})), whence
\[\limsup_{x\to y}\,U^{\gamma_A}(x)\leqslant1,\]
which together with (\ref{ls}) and (\ref{n4'}) results in (\ref{Hcon'}).
\end{proof}

Assume now that $c_*(A)=\infty$, and fix $y\in A^r\cap\partial A$. By virtue of (\ref{Hcon'}), there exist $z_j\in\overline{A}^c$, $j\in\mathbb N$, such that $|z_j-y|<1/j$ and $1\leqslant H_{z_j}<1+1/j$, and applying Theorem~\ref{th1} yields the corollary for these $z_j$ and $q_j:=H_{z_j}$.

\subsection{Proof of Theorem~\ref{th-alt}}\label{pr-alt} Fix $z\in A^u\cup{\overline A}^c$. Under the hypotheses of the theorem, we necessarily have either
\begin{equation}\label{elat1}
q=H_z,
\end{equation}
or
\begin{equation}c_*(A)<\infty,\quad q<H_z.\label{elat2}
\end{equation}

We begin with the case of (\ref{elat1}). As shown in the proof of Theorem~\ref{th1}, then the latter relation in (\ref{sol}) as well as (\ref{RR'}) does indeed hold (see (\ref{Sol}) and (\ref{iec0}), respectively). Therefore, by Theorem~\ref{l-oo'} with $\zeta:=H_z\varepsilon_z$, $\lambda_{A,f_{H_z,z}}$ is uniquely characterized within $\mathcal E^+(A)$ by the equality $U^{\lambda_{A,f_{H_z,z}}}=U^{H_z\varepsilon_z}$ n.e.\ on $A$, or equivalently
\[U^{\lambda_{A,f_{H_z,z}}}_{f_{H_z,z}}=0\quad\text{n.e.\ on $A$},\]
which together with (\ref{RR'}) establishes (iii). To verify (i) and (ii), we first observe from the latter relation in (\ref{sol}), by making use of Section~\ref{some}, (c), that
\begin{equation}\label{elat3}
\lambda_{A,f_{H_z,z}}=H_z\varepsilon_z^A=\bigl(H_z\varepsilon_z^A\bigr)^A.
\end{equation}
Furthermore,
\begin{equation}\label{elat4}
\Lambda_{A,f_{H_z,z}}=\Gamma_{A,H_z\varepsilon_z}=\Gamma_{A,H_z\varepsilon_z^A},
\end{equation}
which is obvious from (\ref{L}), (\ref{RR'}), and (\ref{gammaa}) in view of $U^{\varepsilon_z}=U^{\varepsilon_z^A}$ n.e.\ on $A$ and Lemma~\ref{str-sub}. Applying Theorem~\ref{th-bal2}(i$_4$), resp.\  Theorem~\ref{th-bal1}(iii$_3$), to $\zeta:=H_z\varepsilon_z\in\mathfrak M^+$, resp.\ $\sigma:=H_z\varepsilon_z^A\in\mathcal E^+$, on account of (\ref{elat3}) and (\ref{elat4}) we arrive at (i), resp.\ (ii).

Assuming now that (\ref{elat2}) holds, we define
\[\delta:=q\varepsilon_z^A+L\gamma_A,\quad\text{where } L:=\frac{H_z-q}{H_zc_*(A)}\in(0,\infty).\]
Then obviously $\delta(\mathbb R^n)=1$; hence, in view of (\ref{eqfin}) and (\ref{FenA}),
\begin{equation*}
\delta\in\breve{\mathcal E}^+(A).
\end{equation*}
Besides, by (\ref{ineq1}) with $\zeta:=\varepsilon_z$, (\ref{eqne}), and Lemma~\ref{str-sub}, we have $U^\delta_{f_{q,z}}=L$ n.e.\ on $A$, whence $\delta$-a.e., which implies that, actually,
\begin{equation*}
U^\delta_{f_{q,z}}=L=\int U^\delta_{f_{q,z}}\,d\delta\quad\text{n.e.\ on $A$}.
\end{equation*}
By comparing this with (\ref{ch1}) and (\ref{const}) in Lemma~\ref{l-char}, we get
\begin{equation}\label{elat7}\lambda_{A,f_{q,z}}=\delta,\quad c_{A,f_{q,z}}=L=\frac{H_z-q}{H_zc_*(A)},\end{equation}
cf.\ (\ref{iec}), thereby establishing the former relation in (\ref{sol}) as well as (\ref{R'}).

As $\varepsilon_z^A=(\varepsilon_z^A)^A$ and $\gamma_A=(\gamma_A)^A=\bigl((\gamma_A)^A\bigr)^A$, the latter being derived from (\ref{eqfin'}), $\lambda_{A,f_{q,z}}$ is representable in either of the forms
\begin{equation}\label{f1}
 \lambda_{A,f_{q,z}}=\vartheta^A=(\vartheta^A)^A,
\end{equation}
where
\begin{equation}\label{f2}\vartheta:=q\varepsilon_z+c_{A,f_{q,z}}\gamma_A.\end{equation}
Therefore, by virtue of Theorem~\ref{l-oo'} with $\zeta:=\vartheta$, $\lambda_{A,f_{q,z}}$ is uniquely characterized within $\mathcal E^+(A)$ by the equality $U^{\lambda_{A,f_{q,z}}}=U^\vartheta$ n.e.\ on $A$, or equivalently
\[U^{\lambda_{A,f_{q,z}}}_{f_{q,z}}=c_{A,f_{q,z}}\quad\text{n.e.\ on $A$},\]
the equivalence being obvious from (\ref{eqne}) by utilizing Lemma~\ref{str-sub}. This proves (iii).

It is easy to see, by use of (\ref{L}), (\ref{gammaa}), (\ref{ineq1}), (\ref{eqne}), (\ref{f2}) and Lemma~\ref{str-sub}, that
\begin{equation}\label{f2'}\Lambda_{A,f_{q_z,z}}=\Gamma_{A,\vartheta}=\Gamma_{A,\vartheta^A}.\end{equation}
Now, applying Theorem~\ref{th-bal2}(i$_4$), resp.\  Theorem~\ref{th-bal1}(iii$_3$), to $\zeta:=\vartheta\in\mathfrak M^+$, resp.\ $\sigma:=\vartheta^A\in\mathcal E^+$, on account of (\ref{f1}) and (\ref{f2'}) we arrive at (i), resp.\ (ii).

\subsection{Proof of Theorem~\ref{th2}}\label{pr-th2} We begin by showing that in the case where $q>H_z$, the support $S(\lambda_{A,f_{q,z}})$ is compact. (The solution $\lambda_{A,f_{q,z}}$ does exist, see Theorem~\ref{th1}.)

To this end, we shall first verify that
\begin{equation}\label{i0}
 c_{A,f_{q,z}}<0,
\end{equation}
$c_{A,f_{q,z}}$ being the inner $f_{q,z}$-weighted equilibrium constant (see (\ref{iec})). Noting from Lemma~\ref{l-ext1} that $\lambda_{A,f_{q,z}}$ necessarily equals the extremal measure $\xi:=\xi_{A,f_{q,z}}$, we get
\[c_{A,f_{q,z}}=\int U^{\lambda_{A,f_{q,z}}}_{f_{q,z}}\,d\lambda_{A,f_{q,z}}=\int U^\xi_{f_{q,z}}\,d\xi,\]
whence (\ref{i0}), for $\int U^\xi_{f_{q,z}}\,d\xi<0$ by virtue of (\ref{Cxi}) and (\ref{Cxi0}).

Assume to the contrary that $S(\lambda_{A,f_{q,z}})$ is noncompact. In consequence of (\ref{ch2}) and (\ref{const}), then there is a sequence $(x_j)\subset A$ convergent to $\infty_{\mathbb R^n}$ and such that
\[U^{\lambda_{A,f_{q,z}}}_{f_{q,z}}(x_j)=c_{A,f_{q,z}},\]
whence
\[0<-c_{A,f_{q,z}}\leqslant qU^{\varepsilon_z}(x_j),\]
the former inequality being valid by virtue of (\ref{i0}). Letting now $j\to\infty$, we arrive at a contradiction, for $U^{\varepsilon_z}(x)\to0$ when $x\to\infty_{\mathbb R^n}$. This establishes (i$_1$).

Under the requirements of (ii$_1$), assume first that $q=H_z$. Then, according to the latter relation in (\ref{sol}), $\lambda_{A,f_{q,z}}$ coincides up to a positive factor with the $\alpha$-har\-mon\-ic measure $\varepsilon_z^A$, and (\ref{R}) follows directly from \cite[Theorem~4.1]{Z-bal2}, describing the support of $\varepsilon_z^A$.
If now $q<H_z$ and $c_*(A)<\infty$,\footnote{Recall that in the case where $q<H_z$ and $c_*(A)=\infty$, $\lambda_{A,f_{q,z}}$ fails to exist (Theorem~\ref{th1}).} then, by virtue of (\ref{R'}) and the former relation in (\ref{sol}), $\lambda_{A,f_{q,z}}$ is the linear combination with positive coefficients of the $\alpha$-harmonic measure $\varepsilon_z^A$ and the equilibrium measure $\gamma_A$, and (\ref{R}) follows at once by combining the above-mentioned \cite[Theorem~4.1]{Z-bal2} with \cite[Theorem~7.2]{Z-bal}, the latter describing the support of $\gamma_A$.

\subsection{Proof of Theorem~\ref{th3}}\label{pr-th3} It is seen from (\ref{sol}) and (\ref{elat7}) that for any $z\in\overline{A}^c$, $\lambda_{A,f_{H_z,z}}=H_z\varepsilon_z^A$, whereas in the case where $q<H_z$ and $c_*(A)<\infty$,
\begin{equation*}\lambda_{A,f_{q,z}}=q\varepsilon_z^A+\frac{H_z-q}{H_zc_*(A)}\gamma_A.\end{equation*}
On account of Lemma~\ref{L-Hcon}, both (i$_2$) and (ii$_2$) are therefore reduced to
\[
\lim_{\overline{A}^c\ni z\to y\in\overline{A}^c\cup(A^r\cap\partial A)}\,\varepsilon^A_z(\varphi)=\varepsilon^A_y(\varphi)\quad\text{for all $\varphi\in C_0(\mathbb R^n)$},\]
which indeed holds true by virtue of \cite[Theorem~3.1]{Z-bal2}.

\section{Examples}\label{sec-appl}

In this section, $n=3$ and $\alpha=2$. The aim of the following simple Example~\ref{ex4} is to show that, for given $A$ and $z$, the solution $\lambda_{A,f_{q,z}}$ might be independent of $q$.

\begin{example}\label{ex4}In Problem~\ref{pr} with $A:=S_{0,1}$ and $z:=0$, the (unique) solution $\lambda_{A,f_{q,z}}$ does exist for any $q\in(0,\infty)$ (Theorem~\ref{th1}), the capacity of $S_{0,1}$ being finite. By symmetry reasons, for each $q$, $\lambda_{A,f_{q,z}}$ is uniformly distributed over $S_{0,1}$, hence it equals the equilibrium measure $\gamma_{S_{0,1}}$. Since $I(\gamma_{S_{0,1}})=U^{\gamma_{S_{0,1}}}(0)=1$, we also have
\begin{equation*}
w_{f_{q,z}}(A)=1-2q,\quad c_{A,f_{q,z}}=1-q,
\end{equation*}
the latter equality being in agreement with (\ref{R'})--(\ref{RRR'}) and (\ref{elat7}).
\end{example}

In Examples~\ref{ex1}--\ref{ex3} below, we consider the rotation bodies
\begin{equation*}\Delta_i:=\bigl\{x\in\mathbb R^3: \ 1\leqslant x_1<\infty, \
x_2^2+x_3^2\leqslant\varrho_i^2(x_1)\bigr\},\quad i=1,2,3,\end{equation*}
where
\begin{align*}
\varrho_1(x_1)&:=x_1^{-s}\quad\text{with\ }s\in[0,\infty),\\
\varrho_2(x_1)&:=\exp(-x_1^s)\quad\text{with\ }s\in(0,1],\\
\varrho_3(x_1)&:=\exp(-x_1^s)\quad\text{with\ }s\in(1,\infty),
\end{align*}
and define
\[F_i:=\Delta_i\cup\Delta,\]
$\Delta$ being introduced in Example~\ref{ex} with $\beta\in(1,\infty)$. Then each of $F_i$, $i=1,2,3$, is closed, and $F_i^u$, the set of all $2$-ultrairregular points for $F_i$, is reduced to the origin:
\[F_i^u=\{0\}.\]
Furthermore, $F_1$ is not $2$-thin at infinity, $F_2$ is $2$-thin (but not $2$-ul\-tra\-thin) at infinity, whereas $F_3$ is $2$-ul\-tra\-thin at infinity (cf.\ \cite[Example~2.1]{Z-bal2}).

\begin{example}\label{ex1}For any $z\in\{0\}\cup F_1^c$, the solution $\lambda_{F_1,f_{q,z}}$ to Problem~\ref{pr} with $A:=F_1$ exists if and only if $q\geqslant1$. (This follows from Theorem~\ref{th1}, for $H_{F_1,z}=1$, the set $F_1$ not being $2$-thin at infinity.) Furthermore, if $q=1$, then $\lambda_{F_1,f_{1,z}}=\varepsilon_z^{F_1}$ (Theorem~\ref{th-alt}),
hence $S(\lambda_{F_1,f_{1,z}})$ is noncompact being equal to $\partial F_1$ (Theorem~\ref{th2}(ii$_1$)),
whereas $S(\lambda_{F_1,f_{q,z}})$ is already compact for all $q>1$ (Theorem~\ref{th2}(i$_1$)). We also note that, by virtue of (\ref{RR'}) and (\ref{RRR'}), $c_{F_1,f_{1,z}}=0$, whereas $c_{F_1,f_{q,z}}<0$ for all $q>1$.
\end{example}

\begin{example}\label{ex2} For any $z\in\{0\}\cup F_2^c$, $\lambda_{F_2,f_{q,z}}$ exists if and only if
\[q\geqslant H_z:=1/U^{\gamma_{F_2}}(z),\]
$\gamma_{F_2}$ being the equilibrium measure of $F_2$ (Theorem~\ref{th1}). (This $\gamma_{F_2}$ does exist since $F_2$ is $2$-thin at infinity; though $\gamma_{F_2}(\mathbb R^n)=I(\gamma_{F_2})=\infty$, $F_2$ being of infinite capacity.)
Furthermore, if $q=H_z$, then $\lambda_{F_2,f_{H_z,z}}=H_z\varepsilon_z^{F_2}$ (Theorem~\ref{th-alt}),
hence $S(\lambda_{F_2,f_{H_z,z}})$ is noncompact being equal to $\partial F_2$ (Theorem~\ref{th2}(ii$_1$)),
whereas $S(\lambda_{F_2,f_{q,z}})$ is already compact for all $q>H_z$ (Theorem~\ref{th2}(i$_1$)). Besides, according to (\ref{RR'}) and (\ref{RRR'}), $c_{F_2,f_{H_z,z}}=0$, whereas $c_{F_2,f_{q,z}}<0$ for all $q>H_z$.
\end{example}

\begin{example}\label{ex3} For any $z\in\{0\}\cup F_3^c$ and for any $q\in(0,\infty)$, $\lambda_{F_3,f_{q,z}}$ does exist (Theorem~\ref{th1}), the capacity of $F_3$ being finite; moreover, for all $q\leqslant H_z$, $\lambda_{F_3,f_{q,z}}$ is given by means of (\ref{sol}). Hence, for all $q\leqslant H_z$, $S(\lambda_{F_3,f_{q,z}})$ is noncompact being equal to $\partial F_3$ (Theorem~\ref{th2}(ii$_1$)),
whereas it is already compact for all $q>H_z$ (Theorem~\ref{th2}(i$_1$)). We also note that, by virtue of (\ref{R'})--(\ref{RRR'}),
$c_{F_3,f_{q,z}}$ is strictly positive for all $q<H_z$, strictly negative for all $q>H_z$, and equal to zero otherwise. Moreover, for $q<H_z$, $c_{F_3,f_{q,z}}$ is representable by means of formula (\ref{elat7}).
\end{example}

\section{Acknowledgements} The author thanks Krzysztof Bogdan for helpful discussions on the paper \cite{KB}. 
This work was supported by a grant from the Simons Foundation (1030291, N.V.Z.).

\end{document}